\documentclass[a4paper,11pt]{article}
\usepackage{graphicx} 

\usepackage{amsmath}
\usepackage{amsthm}
\usepackage{amsfonts}
\usepackage{amssymb}
\usepackage{xcolor}
\usepackage{parskip}
\usepackage{comment}
\usepackage{breakurl}
\usepackage[hidelinks]{hyperref}
\usepackage[top=21mm, bottom=23mm, left=22mm, right=22mm]{geometry}
\usepackage{bbm}

\newtheorem{thm}{Theorem}[section]
\newtheorem{prop}[thm]{Proposition}
\newtheorem{lem}[thm]{Lemma}
\newtheorem{cor}[thm]{Corollary}
\newtheorem{conj}[thm]{Conjecture}
\newtheorem{ques}[thm]{Question}

\theoremstyle{definition}

\newtheorem{remark}[thm]{\bf Remark}
\newtheorem{df}[thm]{\bf Definition}

\newtheorem{example}[thm]{\bf Example}

\title{The probability that a random graph is even-decomposable}

\author{Oliver Janzer\thanks{Department of Pure Mathematics and Mathematical Statistics, University of Cambridge, United Kingdom. Research supported by a fellowship at Trinity College. Email: \textbf{oj224@cam.ac.uk}.}
	\and
Fredy Yip\thanks{Trinity College, University of Cambridge, United Kingdom. Email: \textbf{fy276@cam.ac.uk}.}}
\date{}

\begin{document}

\maketitle

\begin{abstract}
    A graph $G$ with an even number of edges is called even-decomposable if there is a sequence $V(G)=V_0\supset V_1\supset \dots \supset V_k=\emptyset$ such that for each $i$, $G[V_i]$ has an even number of edges and $V_i\setminus~V_{i+1}$ is an independent set in $G$. The study of this property was initiated recently by Versteegen, motivated by connections to a Ramsey-type problem and questions about graph codes posed by Alon. Resolving a conjecture of Versteegen, we prove that all but an $e^{-\Omega(n^2)}$ proportion of the $n$-vertex graphs with an even number of edges are even-decomposable. Moreover, answering one of his questions, we determine the order of magnitude of the smallest $p=p(n)$ for which the probability that the random graph $G(n,1-p)$ is even-decomposable (conditional on it having an even number of edges) is at least $1/2$.

    We also study the following closely related property. A graph is called even-degenerate if there is an ordering $v_1,v_2,\dots,v_n$ of its vertices such that each $v_i$ has an even number of neighbours in the set $\{v_{i+1},\dots,v_n\}$. We prove that all but an $e^{-\Omega(n)}$ proportion of the $n$-vertex graphs with an even number of edges are even-degenerate, which is tight up to the implied constant.
\end{abstract}

\section{Introduction}

Given two graphs $G_1$ and $G_2$ on vertex set $[n]$, the sum $G_1+G_2$ is the graph on vertex set $[n]$ whose edge set is the symmetric difference of the edge sets $E(G_1)$ and $E(G_2)$. The set of all graphs on vertex set $[n]$ is an $\mathbb{F}_2$-vector space under this operation, and will be denoted $\mathbb{F}_2^{E(K_n)}$. 

Given a graph $H$, not necessarily on vertex set $[n]$, we say that a graph $G$ on vertex set $[n]$ is an \emph{image} of $H$ in $\mathbb{F}_2^{E(K_n)}$ if there is some $S\subseteq [n]$ such that $G[S]$ is isomorphic to $H$ and all vertices in $[n]\backslash S$ are isolated in $G$. 

For a collection of graphs $\mathcal{H}$, an $\mathcal{H}$-graph code (on $n$ vertices) is a collection $\mathcal{C}\subseteq\mathbb{F}_2^{E(K_n)}$ of graphs on vertex set $[n]$ such that $\mathcal{C} + \mathcal{C} = \{C_1+C_2 \hspace{1mm} | \hspace{1mm} C_1, C_2\in\mathcal{C}\}\subseteq\mathbb{F}_2^{E(K_n)}$ does not contain images of any graph $H\in\mathcal{H}$.

Alon et al. \cite{alon2023structured,alon2024graph} initiated the systematic study of the maximum density of $\mathcal{H}$-graph codes. Let $d_{\mathcal{H}}(n)$ be the maximum of $|\mathcal{C}|/2^{\binom{n}{2}}$ over all $\mathcal{H}$-graph codes $\mathcal{C}$ on $n$ vertices. In \cite{alon2023structured}, the growth of $d_\mathcal{H}(n)$ was studied for the family of disconnected graphs, the family of not 2-connected graphs, the family of non-Hamiltonian graphs and the family of graphs containing or avoiding a spanning star. 

This function had in fact been studied earlier for specific families $\mathcal{H}$. The case where $\mathcal{H}$ is the family of all cliques was discussed in the comments of a blog post by Gowers \cite{gowers}. Kalai asked whether in this case we have $d_\mathcal{H}(n) = o(1)$, and this question is still open. On the other hand, $d_\mathcal{H}(n)$ has been exactly determined for the family of graphs with independence number at most two \cite{ellis2012triangle} and the family of graphs with independence number at most three \cite{berger2023k4}. 

In \cite{alon2024graph}, Alon studied $d_{\mathcal{H}}(n)$ in the case where $\mathcal{H} = \{H\}$, and (among other results) determined the order of magnitude for stars and matchings.
Every lower bound construction he used is a \emph{linear graph code}, that is, a graph code which forms a linear subspace in $\mathbb{F}_2^{E(K_n)}$. Motivated by this, Alon proposed the study of the maximum possible density of a linear $\{H\}$-graph code. Note that if $\mathcal{C}$ is a linear subspace of $\mathbb{F}_2^{E(K_n)}$, then $\mathcal{C}+\mathcal{C}=\mathcal{C}$, so a linear $\{H\}$-graph code is a linear subspace $\mathcal{C}\subseteq \mathbb{F}_2^{E(K_n)}$ which does not contain any image of $H$. Let $d_H^\text{lin}(n)$ denote the maximum of $|\mathcal{C}|/2^{\binom{n}{2}}$ over all linear $\{H\}$-graph codes $\mathcal{C}$ on $n$ vertices.

Alon \cite{alon2024graph} observed the equivalence of the asymptotic behaviour of $d^\text{lin}_H(n)$ with the following Ramsey-type problem. In an edge-colouring of $K_n$, a copy of $H$ is \textit{even-chromatic} if every colour appears on an even number of edges in it. Let $r_H(n)$ be the minimum number of colours needed to avoid all even-chromatic copies of $H$ in $K_n$. It can be shown (see \cite{versteegen2023upper} for a proof) that if $r_H(n)$ is sufficiently large, then
\begin{equation*}
    \frac{1}{r_H(n)^{e(H)+1}}\leq d^\text{lin}_H(n)\leq \frac{1}{r_H(n)}. 
\end{equation*}

If $H$ has an odd number of edges, then it is not hard to see that for all sufficiently large $n$, we have $d_H^{\textrm{lin}}(n)=d_H(n)=1/2$ (the extremal construction being the collection of all graphs on vertex set $[n]$ with an even number of edges), and trivially $r_H(n)=1$.
Hence, in what follows we will only be interested in graphs $H$ with an even number of edges. 

Alon \cite{alon2024graph} observed that if $e(H)$ is even, then it follows from the best bound on multicolour Ramsey numbers that $r_H(n)=\Omega(\log n/\log \log n)$. Improving this, Versteegen \cite{versteegen2023upper} showed that for some absolute constant $c$, we have $r_H(n)\geq c\log n$ for all graphs $H$ with an even number of edges. On the other hand, it was shown in \cite{cameron2023new} that $r_{K_4}(n)=n^{o(1)}$, and it was proved in \cite{ge2023new} and \cite{heath2023edge} (independently) that $r_{K_5}(n)=n^{o(1)}$. At the same time, it is easy to construct graphs for which $r_H(n)$ grows polynomially. 
This naturally raises the following question, which was considered by Versteegen \cite{versteegen2023upper}.
\begin{ques}
    For which graphs $H$, with an even number of edges, is $r_H(n) = n^{o(1)}$? 
\end{ques}

To address this question, Versteegen introduced the following class of graphs.

\begin{df}[Even-decomposable graphs] \label{def:decomposable}
    A graph $H$ is \emph{even-decomposable} if there is a sequence $V(H) = V_0\supset V_1\supset\cdots \supset V_k = \emptyset$ such that for each $0\leq i\leq k-1$, $H[V_i]$ has an even number of edges and $V_{i}\backslash V_{i+1}$ is an independent set in~$H$.
\end{df}

Note that any even-decomposable graph has an even number of edges.

\begin{df}[Even-non-decomposable graphs]
    If a graph has an even number of edges, but is not even-decomposable, we call it \emph{even-non-decomposable}.
\end{df}

Observe that $K_4$ is even-non-decomposable. 
Versteegen \cite{versteegen2023upper} conjectured that $r_H(n)$ has polynomial growth if and only if it is even-decomposable.

\begin{conj}[Versteegen \cite{versteegen2023upper}] \label{conj:poly if even decomp}
    $r_H(n) = n^{\Omega(1)}$ if and only if $H$ is even-decomposable. 
\end{conj}

He proved that if $H$ is even-decomposable, then indeed $r_H(n)$ has polynomial growth, verifying the `if' direction of Conjecture \ref{conj:poly if even decomp}.

He also showed that almost all graphs with an even number of edges are even-decomposable. Here and below $G(n,p)$ denotes the binomial random graph on $n$ vertices with edge probability $p$.
\begin{thm}[Versteegen \cite{versteegen2023upper}]\label{versteegenevendecompthm}
    If $H\sim G(n, \frac{1}{2})$, then
    $$\mathbb{P}(\text{$H$ is even-non-decomposable}) = e^{-\Omega(\sqrt{\log n})}.$$ 
\end{thm}
Versteegen conjectured the following much stronger bound for the proportion of even-non-decomposable graphs.

\begin{conj}[Versteegen \cite{versteegen2023upper}] \label{e^(n^2) conj}
    If $H\sim G(n, \frac{1}{2})$, then
    $$\mathbb{P}(\text{$H$ is even-non-decomposable}) = e^{-\Omega(n^2)}.$$ 
\end{conj}

Our first result is a proof of Conjecture \ref{e^(n^2) conj}.

\begin{thm} \label{e^(n^2)}
    If $H\sim G(n, \frac{1}{2})$, then
    $$\mathbb{P}(\text{$H$ is even-non-decomposable}) = e^{-\Omega(n^2)}.$$ 
\end{thm}

This bound is trivially tight up to the value of the implied constant (e.g. since $K_4$ is even-non-decomposable).

Using Versteegen's result \cite{versteegen2023upper} that if $H$ is even-decomposable, then $r_H(n)=n^{\Omega(1)}$ and Alon's observation \cite{alon2024graph} that $r_H(n)=n^{\Omega(1)}$ if and only if $d_H^{\text{lin}}(n)=n^{-\Omega(1)}$, we obtain the following corollary about graph codes and the Ramsey-type problem discussed above.

\begin{cor}
    There is a positive absolute constant $c$ such that for all but an at most $e^{-cm^2}$ proportion of $m$-vertex graphs $H$ with an even number of edges, we have $d_H^{\text{lin}}(n)\leq n^{-\Omega(1)}$ and $r_H(n)\geq n^{\Omega(1)}$.
\end{cor}

Versteegen noted that, intuitively, the sparser a graph is, the larger its independent sets are, hence the easier it should be to find even-decompositions for it. On the other hand, he pointed out that even very dense graphs can be even-decomposable and that the method used to prove his Theorem~\ref{versteegenevendecompthm} is specific to the case $p=1/2$. Motivated by these observations, he asked to determine the threshold where even-non-decomposable graphs become dominant.

\begin{ques}[Versteegen \cite{versteegen2023upper}] \label{G(n, p) conj}
    What is the minimal $p(n)$ such that when $H\sim G(n, p(n))$, $$\mathbb{P}(\text{$H$ is even-non-decomposable} \hspace{1mm} | \hspace{1mm} \text{$e(H)$ is even}) \geq \frac{1}{2}?$$ 
\end{ques}

We answer this question by showing that the threshold is $p(n) = 1 - \Theta(\frac{1}{n})$.

\begin{thm} \label{G(n, p)}
    There exists an absolute constant $c_1>0$, such that for any $p\leq 1 - \frac{c_1}{n}$ and $H\sim G(n, p)$, we have $$\mathbb{P}(\text{$H$ is even-non-decomposable}) = o(1).$$ 
\end{thm}

\begin{prop} \label{G(n, p) triv dir}
    There exists an absolute constant $c_2>0$, such that for any $p\geq 1 - \frac{c_2}{n}$ and $H\sim G(n, p)$, we have $$\mathbb{P}(\text{$H$ is even-decomposable}) = o(1).$$ 
\end{prop}

Alon \cite{Alonpersonal} asked whether a random graph also has the following stronger (and slightly more natural) property.

\begin{df}[Even-degenerate graphs] \label{def:degenerate}
    A graph is \emph{even-degenerate} if there is an ordering $v_1,v_2,\dots,v_n$ of its vertices such that for each $1\leq i\leq n-2$, $v_i$ has an even number of neighbours in the set $\{v_{i+1},\dots,v_n\}$.
\end{df}

Note that here we do not require the graph to have an even number of edges, and, accordingly, we do not require that $v_{n-1}v_n$ is a non-edge (which would correspond to $v_{n-1}$ having an even number of neighbours in the set $\{v_n\}$). We point out that for an ordering satisfying the above condition, the parity of $e(H[\{v_i, \cdots, v_n\}])$ is the same as the parity of $e(H)$ for each $1\leq i\leq n-1$.

Observe that every even-degenerate graph with an even number of edges is even-decomposable. Indeed, given an even-degenerate graph $H$ with an ordering $v_1,\dots,v_n$ of its vertices as in Definition~\ref{def:degenerate}, we may take $V_i=\{v_{i+1},\dots,v_n\}$ for each $0\leq i\leq n$, and these sets satisfy the conditions in Definition \ref{def:decomposable} provided that $e(H)$ is even.

We answer Alon's question in the following tight form.

\begin{thm} \label{thm:degenerate}
    If $H\sim G(n,\frac{1}{2})$, then
    $$\mathbb{P}(\text{$H$ is not even-degenerate})=e^{-\Theta(n)}.$$
\end{thm}

\textbf{Organization of the paper.} In Section \ref{sec:decomposable}, we prove our results about the probability that a random graph is even-decomposable. In Subsection \ref{sec:prelim}, we collect some preliminary observations, and briefly describe our proof strategy. In Subsection \ref{1/2 section}, we prove Theorem \ref{e^(n^2)}. In Subsection~\ref{general p section}, we prove Theorem \ref{G(n, p)} and Proposition \ref{G(n, p) triv dir}. In Section \ref{sec:degenerate}, we prove Theorem \ref{thm:degenerate} concerning the probability that a random graph is even-degenerate.

\textbf{Notation.} For a non-negative integer $n$, let $[n] = \{1, \cdots, n\}$. For a graph $G$, we denote by $V(G)$ and $E(G)$ the vertex set and edge set of $G$, respectively, and we denote by $v(G)$ and $e(G)$ the number of vertices and edges, respectively. For a set $A$ of vertices of $G$, let $e(A) = e(G[A])$. For $i\neq j\in V(G)$, let $G_{ij}$ be $1$ if $ij$ is an edge in $G$ and $0$ otherwise. For a vertex set $V$, we denote by $G(V, p)$ the random graph on vertex set $V$ where each pair has probability $p$ of being an edge, independently. Throughout the paper we often omit the rounding functions of real numbers for the clarity of presentation.

\section{Even-decomposable graphs} \label{sec:decomposable}

\subsection{Preliminaries} \label{sec:prelim}

Throughout the paper, we will use the Chernoff bound for the tail of the binomial distribution $\mathrm{Bin}(n, p)$. 
\begin{lem}[Chernoff bound]
    Let $X\sim \mathrm{Bin}(n, p)$, then
    \begin{align*}
        &\mathbb{P}\left(X >(1+\delta) pn\right)\leq e^{-\frac{\delta^2}{2+\delta}np} && \text{for } \delta\geq0, \\
        &\mathbb{P}\left(X <(1-\delta) pn\right)\leq e^{-\frac{\delta^2}{2}np} && \text{for } 0\leq\delta\leq1. 
    \end{align*}
\end{lem}

Note that both are $e^{-\Omega(np)}$ for $\delta$ bounded away from $0$.

Recall that an even-decomposition consists of repeated removals of independent sets until we obtain the empty graph, such that at each stage the remaining graph has an even number of edges. The following notion will refer to a removal of vertices that is allowed in an even decomposition.

\begin{df}[Admissible removals]
    The removal of a set $S$ of vertices from a graph $G$ is \emph{simple admissible} if $S$ is an independent set in $G$ and the number of edges between $S$ and $V(G)\setminus S$ is even. A removal of a (not necessarily independent) set of vertices from a graph is \emph{admissible} if it may be facilitated by a sequence of simple admissible removals. 
\end{df}

Then a graph with an even number of edges is even-decomposable if and only if we can admissibly remove all of its vertices.

We note that we may simply admissibly remove a vertex of even degree by itself or a pair of non-adjacent vertices of odd degree. In fact, it is easy to see that any simple admissible (and hence also any admissible) removal can be facilitated by a sequence of such removals (though we will not make use of this observation).

\begin{remark}
    In what follows, when we perform a sequence of deletions of vertices in a graph, the degree of a vertex shall always refer to the degree in the remaining graph (as opposed to its degree in the original graph).
\end{remark}

The following lemma shows that given any set of vertices, we can keep (admissibly) removing vertices until we get stuck with a clique in which every vertex has odd degree (with respect to the entire remaining graph).

\begin{lem}[Greedy Admissible Removal]\label{greedy}
    For any set $S$ of vertices in a graph $G$, we may admissibly remove vertices from $S$, such that the remaining vertices in $S$ form a clique and all have odd degree in the remaining graph. 
\end{lem}

\begin{proof}
    We greedily perform simple admissible removals of vertices from $S$. Once no further simple admissible removals are possible, there cannot be any vertices of even degree left in $S$ as they could be simply admissibly removed by themselves. Similarly there cannot be two non-adjacent vertices of odd degree left in $S$, as such a pair could also be simply admissibly removed.
\end{proof}

\begin{remark} \label{rem:k4free}
From Lemma \ref{greedy} it follows that any graph with an even number of edges and without copies of $K_4$ is even-decomposable. Indeed, applying Lemma \ref{greedy} to the entire vertex set of such a graph results in $K_m$ for some $m\leq 3$. As $K_m$ must have an even number of edges and all vertices have odd degree, we must have $m = 0$ and the graph is even decomposed.
\end{remark}

Lemma \ref{greedy} shows that if we greedily keep applying simple admissible removals to a graph, we can only get stuck with a clique.

\textbf{Proof strategy.} Our proof strategy for both Theorem \ref{e^(n^2)} and Theorem \ref{G(n, p)} is to put aside linearly many vertex-disjoint induced copies of a carefully chosen small `gadget' graph. The exact choice of this gadget graph will depend on the probability $p$ that we are considering, but it will be designed to be able to `absorb' (via an admissible removal of a set of vertices) a vertex in a clique, so that, using these gadget graphs we can destroy any large clique that we may be stuck with after applying the greedy removal (Lemma \ref{greedy}). We will prove both main theorems by identifying deterministic conditions (see Lemma \ref{lem:deterministic 1/2} and Lemma \ref{lem:deterministic dense} for Theorem \ref{e^(n^2)} and Theorem~\ref{G(n, p)}, respectively) involving the existence of these gadget graphs as well as a bound on the size of the largest clique in the graph, which together imply even decomposability, and then arguing that in a random graph those conditions are satisfied with very high probability.

\subsection{Uniform distribution \texorpdfstring{$G(n, \frac{1}{2})$}{}}\label{1/2 section}

In this subsection, we prove Theorem \ref{e^(n^2)}.

The following lemma describes how a gadget graph is used to absorb a vertex, and captures one of the key ideas in our proof.

\begin{lem} \label{1/2 first lem}
    Let $G$ be a graph. Let $A$ and $B$ be disjoint independent sets in $G$ of size at least two. Let $v$ be a vertex outside $A\cup B$. Assume also that every $a\in A$ sends both an edge and a non-edge to $B$, and every $b\in B$ sends both an edge and a non-edge to $A$. Then we may admissibly remove vertices from $A\cup B\cup \{v\}$ so that $v$ is removed but at most two vertices are removed from each of $A, B$. 
\end{lem}

\begin{proof}
    We start with the following claim.

    \noindent \textit{Claim.} If there is some $x\in A\cup B$ of even degree such that $xv$ is an edge, or there is some $x\in A\cup B$ of odd degree such that $xv$ is a non-edge, then we may admissibly remove $v$ and at most one vertex from $A\cup B$.

    \noindent \textit{Proof of Claim.} If $v$ has even degree, we may admissibly remove $\{v\}$. Hence, we may assume that $v$ has odd degree. Let $x\in A\cup B$. If $x$ has even degree and $xv$ is an edge, then we may (simply) admissibly remove $\{x\}$ first, and then $\{v\}$. If $x$ has odd degree and $xv$ is a non-edge, then we may (simply) admissibly remove $\{x,v\}$. $\Box$
    
    Using the claim, we can assume that
    
    $(\ast)$ for each $x\in A\cup B$, $xv$ is an edge if and only if $x$ has odd degree. 
    
    If there is any $x\in A\cup B$ with even degree, without loss of generality, assume that $x\in A$. Let $xb$ be an edge for some $b\in B$. We may admissibly remove $\{x\}$, which flips the parity of the degree of $b$, after which $bv$ is an edge if and only if $b$ has even degree. Hence, by the proof of the claim above, we may admissibly remove $v$ and at most one vertex in $A\cup B$, completing the proof of the lemma. 

    Hence, we may assume that all vertices in $A\cup B$ have odd degree, and hence, by $(\ast)$, $xv$ is an edge for all $x\in A\cup B$. We may also assume that $v$ has odd degree (otherwise we can remove $\{v\}$). Let $ab_1$ be an edge and $ab_2$ be a non-edge, where $a\in A$, $b_1, b_2\in B$. Removing first $\{b_1, b_2\}$, then $\{a\}$, and, finally, $\{v\}$, we have performed a suitable admissible removal.
\end{proof}

Copies of the following graph will be used to `absorb' cliques via Lemma \ref{1/2 first lem}.

\begin{df}
    Let $F$ be the union of two vertex-disjoint copies of $K_{5, 5}$. 
\end{df}

\begin{lem} \label{1step}
    Let $G$ be a graph. Let $R\subset V(G)$ be a set of vertices such that $G[R]$ is isomorphic to $F$, and let $C$ be a clique consisting of $m$ vertices, which is disjoint from $R$. Then we may admissibly remove vertices from $C\cup R$ until it becomes a clique of at most $\max(m-1, 2)$ vertices. 
\end{lem}

\begin{proof}
    Since $G[R]$ is isomorphic to $F$, there are disjoint sets $T_1,T_2,S_1,S_2$ of size $5$ in $R$ such that $T_1$ is complete to $S_1$, $T_2$ is complete to $S_2$, and there are no other edges within $G[T_1\cup T_2\cup S_1\cup S_2]$. Let $A = T_1\cup T_2$ and $B = S_1\cup S_2$. We note that $A, B$ are independent sets, and however we remove up to four vertices from each of $A$ and $B$, every remaining vertex in $A$ will send both an edge and a non-edge to the remaining subset of $B$ and vice versa. 

    \noindent \textit{Claim.} We may admissibly remove a set $Q\subset A\cup B\cup C$ containing at least $\min(3,m)$ vertices from~$C$.

    \noindent \textit{Proof of Claim.} If $m=0$, we are already done. Else, choose some $v\in C$ and apply Lemma \ref{1/2 first lem} to $A$, $B$ and $v$. Let $A'\subset A$ and $B'\subset B$ be the remaining subsets, and note that $|A\setminus A'|,|B\setminus B'|\leq 2$. If $m=1$, we are already done, else we can apply Lemma \ref{1/2 first lem} again with $A'$, $B'$ and some $v'\in C\setminus \{v\}$ in place of $A$, $B$ and $v$. Let $A''\subset A$ and $B''\subset B$ be the remaining subsets, and note that $|A\setminus A''|,|B\setminus B''|\leq 4$. If $m=2$, we are already done, else we can apply Lemma \ref{1/2 first lem} again with $A''$, $B''$ and some $v''\in C\setminus \{v,v'\}$ in place of $A$, $B$ and $v$, completing the proof of the claim. $\Box$
    
    After applying the claim above, we have at most $\max(m - 3, 0)$ remaining vertices in $C$.
    We now remove vertices greedily (via Lemma \ref{greedy}) from the remainders of $A$ and $B$. As they are independent sets, at most one vertex will remain from each set. A total of at most $\max(m - 1, 2)$ vertices now remain in $C\cup R$, to which we apply Lemma \ref{greedy} once more so that the remaining vertices form a clique. 
\end{proof}

We now identify deterministic conditions which imply even decomposability.

\begin{lem} \label{lem:deterministic 1/2}
    Let $G$ be a graph with an even number of edges. Assume that there exists some positive integer $t$ such that $G$ contains at least $t$ vertex-disjoint induced copies of $F$ and $G$ has no clique of size greater than $t$. Then $G$ is even-decomposable.
\end{lem}

\begin{proof}
    Let $R_1, \dots, R_{t}$ be the vertex sets of $t$ vertex-disjoint induced copies of $F$. We perform the greedy algorithm (Lemma \ref{greedy}) on the other vertices in $G$ to obtain a clique $C_1$, which must have at most $t$ vertices by the clique number bound. 
    
    We apply Lemma \ref{1step} to $R_1$ and $C_1$ to obtain a clique $C_2$ of at most $\max(t - 1, 2)$ vertices. We repeat this using $R_2, \dots, R_{t}$, until we have only a clique of at most $\max(t - t, 2) = 2$ vertices remaining, which is now clearly even-decomposable, as $G$ has an even number of edges. 
\end{proof}

We are now ready to prove Theorem \ref{e^(n^2)} in the following slightly stronger form.

\begin{thm} \label{p average thm}
    If $\frac{1}{10^6}\leq p\leq 1 - \frac{1}{10^6}$ and $G\sim G(n, p)$, then
    \begin{equation*}
        \mathbb{P}(G\text{ is even-non-decomposable})= e^{-\Omega(n^2)}. 
    \end{equation*}
\end{thm}

\begin{proof}
    By Lemma \ref{lem:deterministic 1/2}, it suffices to prove that there exists some positive integer $t$ such that with probability $1-e^{-\Omega(n^2)}$, $G$ contains at least $t$ vertex-disjoint induced copies of $F$, and $G$ has no clique of size $t$.
    
    We first lower bound the number of vertex-disjoint induced copies of $F$ in $G$. Let $S_1, \dots, S_k$ be a maximal collection of edge-disjoint $K_{20}$s in $K_n$. It is easy to see, e.g. by Tur\'an's theorem, that $k\geq c_1n^2$ for some absolute constant $c_1>0$ (as long as $n\geq 20$). As the $S_i$ are edge-disjoint, the events that (the vertex set of) $S_i$ induces a copy of $F$ in $G$ are independent over all $i\in [k]$, each occurring with probability $q$, where $q=\Omega(1)$ since $10^{-6}\leq p\leq 1-10^{-6}$. Hence, the total number of $S_i$ which induce a copy of $F$ in $G$ dominates $\mathrm{Bin}(c_1n^2,q)$, which, by the Chernoff bound, is at least $c_2n^2$ with probability $1 - e^{-\Omega(n^2)}$, for some absolute constant $c_2>0$. Thus, with this probability there exists a collection $\mathcal{U}$ of $c_2n^2$ induced copies of $F$ in $G$, any two of which share at most one vertex. Take a maximal subcollection $\mathcal{V}\subset \mathcal{U}$ of vertex-disjoint induced copies of $F$. The maximality implies that every induced copy of $F$ in $\mathcal{U}$ shares at least one vertex with some element in $\mathcal{V}$. As one vertex may be shared by at most $\frac{n-1}{19}$ elements of $\mathcal{U}$ (since no two elements of $\mathcal{U}$ share more than one vertex), we have $|\mathcal{V}| \geq \frac{1}{20}\frac{19}{n-1}|\mathcal{U}|\geq c_3n$, for some absolute constant $c_3>0$. Hence with probability $1 - e^{-\Omega(n^2)}$, $G$ contains at least $c_3n$ vertex-disjoint induced copies of $F$.
    
    We upper bound the clique number of $G$. The expected number of copies of $K_{ c_3n}$ in $G$ is $\binom{n}{ c_3n} p^{\binom{ c_3n}{2}} = e^{-\Omega(n^2)}$. Hence, by Markov's inequality, with probability $1 - e^{-\Omega(n^2)}$, there are no copies of $K_{ c_3n}$ in~$G$. Thus, we may take $t=c_3n$ and apply Lemma \ref{lem:deterministic 1/2}.
\end{proof}

\subsection{General case \texorpdfstring{$G(n, p)$}{}}\label{general p section}

In this subsection we prove Theorem \ref{G(n, p)} and Proposition \ref{G(n, p) triv dir}.

Theorem \ref{p average thm} already shows that if $p$ is bounded away from $0$ and $1$, then the probability that $G(n,p)$ is even-non-decomposable is low.
We will need to extend this to the case where $p$ is very small or very large. In these cases, we use the path of three edges $P_3$ in place of the 20-vertex graph $F$. The analogs of Lemma \ref{1step} with $P_3$ in place of $F$ are Lemmas \ref{dense stage 2} and \ref{dense stage 3} in the dense regime, and Lemmas \ref{sparse stage 2} and \ref{sparse stage 3} in the sparse regime.\footnote{We remark that using $P_2$ in these lemmas would also suffice, but it would make certain parts of the argument more complicated.} Note that, compared to Lemma \ref{1step}, these lemmas have a crucial additional assumption that $R$ is either complete to $C$ (to be used in the dense regime) or sends no edges at all to $C$ (to be used in the sparse regime), without which they actually do not hold.  

\begin{lem} \label{dense stage 2}
    Let $G$ be a graph. Let $R$ be the vertex set of an induced copy of $P_3$ in $G$, and let $C$ be a clique consisting of $m\geq 2$ vertices of odd degree in $G$, disjoint from $R$. Assume that $R$ is complete to $C$. Then we may admissibly remove vertices from $C\cup R$ until it forms a clique of at most $m-1$ vertices consisting of odd degree.
\end{lem}

\begin{proof}
    Let $a,b,c,d$ be the vertices of the induced copy of $P_3$ in $G[R]$, in the natural order. We prove the lemma by considering each of the 10 cases (by symmetry) of the parity of the degrees of these four vertices. Below, we use $e$ or $o$ to denote when the degree of a vertex in $(a, b, c, d)$ is even or odd, respectively. (So, for example, $(e,e,o,e)$ means that $a,b,d$ have even degree and $c$ has odd degree.) We use $k_1,k_2$ to denote arbitrary distinct elements of $C$.
    \begin{itemize}
        \item $(e, e, e, e)$: Remove $b, k_1, c, k_2, \{a, d\}$.
        \item $(e, e, e, o)$: Remove $b, k_1, c, k_2, d$. 
        \item $(e, e, o, e)$: Remove $b, k_1, a, k_2, c$. 
        \item $(e, e, o, o)$: Remove $b, k_1, a, k_2, c$. 
        \item $(o, o, o, o)$: Remove $\{a, d\}, b, k_1, c, k_2$. 
        \item $(o, o, o, e)$: Remove $d, k_1, b, k_2, a$. 
        \item $(o, o, e, o)$: Remove $c, k_1, a, k_2, d$. 
        \item $(o, e, o, e)$: Remove $d, k_1, a, k_2, c$.
        \item $(o, e, e, o)$: Remove $b, k_1, c, k_2, a$.
        \item $(e, o, o, e)$: Remove $a, k_1, c, k_2, \{b, d\}$.
    \end{itemize}
    For example, in the first case above, $a, b, c, d$ initially have even degree, and all vertices in $C$ have odd degree. We admissibly remove $\{b\}$, after which $a, c$ have odd degree, $d$ has even degree, and all vertices in $C$ have even degree. We admissibly remove $\{k_1\}$ for some $k_1\in C$, after which $a, c$ have even degree, $d$ has odd degree, and all vertices in $C$ have odd degree. We admissibly remove $\{c\}$, after which $a, d$ have even degree, and all vertices in $C$ have even degree. We admissibly remove $\{k_2\}$ for some $k_2\in C$, after which $a, d$ have odd degree, and all vertices in $C$ have odd degree. And we finish by admissibly removing $\{a, d\}$.

    In each of the $10$ cases, we have at most $m-1$ vertices remaining in $C\cup R$ after the removals. Using further removals if necessary, applying the greedy algorithm (Lemma \ref{greedy}), we can achieve that the remaining vertices in $C\cup R$ form a clique consisting of at most $m-1$ vertices of odd degree. 
\end{proof}

\begin{lem} \label{dense stage 3}
    Let $G$ be a graph. Let $R$ be the vertex set of an induced copy of $P_3$ in $G$, and let $C$ be a clique consisting of at most two vertices of odd degree in $G$, disjoint from $R$. Assume that $R$ is complete to $C$. Then we may admissibly remove vertices from $C\cup R$ such that the entire set $C$ is removed and the remainder of $R$ forms a clique consisting of at most two vertices of odd degree. 
\end{lem}

\begin{proof}
    We consider again each of the 10 cases (by symmetry) of the parity of the degrees of the vertices in $R$, and follow the same admissible removal sequence as in Lemma \ref{dense stage 2} until all vertices in $C$ have been removed.

    Once all vertices in $C$ have been removed, we may apply the greedy algorithm (Lemma \ref{greedy}) to the remainder of $R$ to ensure that a clique consisting of at most two vertices of odd degree remains, as the largest clique in $P_3$ has two vertices.     
\end{proof}

The next two lemmas are the counterparts of Lemma \ref{dense stage 2} and Lemma \ref{dense stage 3} in the sparse setting.

\begin{lem} \label{sparse stage 2}
    Let $G$ be a graph. Let $R$ be the vertex set of an induced copy of $P_3$ in $G$, and let $C$ be a clique consisting of $m\geq 2$ vertices of odd degree in $G$, disjoint from $R$. Assume that $R$ sends no edges to $C$. Then we may admissibly remove vertices from $C\cup R$ until it forms a clique consisting of at most $\max(m-1, 2)$ vertices of odd degree. 
\end{lem}

\begin{proof}
    If there is a vertex $t\in R$ of odd degree, admissibly remove $\{t, k\}$, where $k\in C$. Now there are no cliques of size greater than $\max(m-1, 2)$ remaining in $C\cup R$ and greedily removing (see Lemma \ref{greedy}) from $C\cup R$ suffices. 
    
    Else, let $a,b,c,d$ be the vertices of the induced copy of $P_3$ in $G[R]$, in the natural order, all of which have even degree. We may admissibly remove $\{a, c\}, \{d, k_1\}, k_2$, where $k_1, k_2\in C$. $m - 1$ vertices now remain in $C\cup R$, and greedily removing (see Lemma \ref{greedy}) from $C\cup R$ suffices. 
\end{proof}

\begin{lem} \label{sparse stage 3}
    Let $G$ be a graph. Let $R$ be the vertex set of an induced copy of $P_3$ in $G$, and let $C$ be a clique consisting of at most two vertices of odd degree in $G$, disjoint from $R$. Assume that $R$ sends no edges to $C$. Then we may admissibly remove vertices from $C\cup R$ such that the entire set $C$ is removed and the remainder of $R$ forms a clique of at most two vertices of odd degree. 
\end{lem}

\begin{proof}
    Once $C$ is removed, we may apply the greedy algorithm (Lemma \ref{greedy}) to the remainder of $R$ to ensure that a clique of at most two vertices of odd degree remains, as the largest clique in $P_3$ has two vertices. 

    To remove $C$, if any $t\in R$ has odd degree, we may admissibly remove $\{t, k_1\}, k_2$, where $k_1, k_2\in C$, stopping at any stage when all vertices in $C$ have already been removed. 

    Else, let $a,b,c,d$ be the vertices of the induced copy of $P_3$ in $G[R]$, in the natural order, all of which have even degrees. We may remove $C$ by admissibly removing $\{a, c\}, \{d, k_1\}, k_2$, where $k_1, k_2\in C$, again stopping at any stage when all vertices in $C$ have been removed. 
\end{proof}

\begin{lem} \label{lem:deterministic dense}
    Let $G$ be a graph on $n$ vertices with an even number of edges. Assume that $G$ contains at least $\frac{n}{100}$ vertex-disjoint induced copies of $P_3$, $G$ has no copies of $K_{\frac{n}{200}}$, and all vertices have degree at most $\frac{n}{10^5}$ in the complement $\overline{G}$. Then $G$ is even-decomposable.
\end{lem}

\begin{proof}
    Let $\mathcal{H}$ be a collection of $n/100$ vertex-disjoint induced copies of $P_3$ in $G$. 

    We even decompose $G$ in three stages. 

    \begin{itemize}
        \item Stage 1: Perform greedy admissible removals (Lemma \ref{greedy}) on $V(G)\backslash \bigcup_{H\in \mathcal{H}} V(H)$. Let us denote by $C_0$ the remaining clique, which consists of vertices of odd degree. Then $|C_0|\leq\frac{n}{200}$ since $G$ has no clique of size $n/200$.
        
        \item Stage 2: Let $\mathcal{H}_0=\mathcal{H}$. We apply admissible removals so that after $i\leq n/200$ steps, the remaining vertex set is the disjoint union of $\bigcup_{H\in \mathcal{H}_i} V(H)$ and $C_i$, where  $\mathcal{H}_i\subseteq \mathcal{H}_0$ has size $|\mathcal{H}_0|-i$, and $C_i$ is a clique of size at most $|C_0|-i$ which consists of vertices of odd degree.

        If for some $i$, we have $|C_{i-1}|\leq 2$, move onto stage 3. Else, choose some $C\subseteq C_{i-1}$ with $|C| = 3$. Vertices from $C$ send non-edges to the vertex set of at most $\frac{3n}{10^5}$ many $H\in \mathcal{H}_{i-1}$ by the assumption on the degrees in $\overline{G}$. Hence, as $|\mathcal{H}_{i-1}|=|\mathcal{H}_0|-(i-1)\geq n/100-n/200= n/200$, we may choose some $H\in \mathcal{H}_{i-1}$ whose vertex set $R$ is complete to $C$. Now apply Lemma \ref{dense stage 2} to $R$ and $C$, and let the clique within $R\cup C$ obtained after the admissible removal described by the lemma be $C'_{i}$. Note that $|C'_{i}|\leq 2$. Applying greedy admissible removals (Lemma~\ref{greedy}) to the set $C_{i}'\cup (C_{i-1}\setminus C)$, we obtain a clique $C_{i}$ consisting of vertices of odd degree. Let $\mathcal{H}_{i}=\mathcal{H}_{i-1}\setminus \{H\}$. Clearly, we have $|\mathcal{H}_{i}|=|\mathcal{H}_{i-1}|-1=|\mathcal{H}_0|-i$ and $|C_{i}|=|C_{i}'|+|C_{i-1}|-|C|\leq |C_{i-1}|-1\leq |C_0|-i$, and the remaining vertex set is the disjoint union of $\bigcup_{H\in \mathcal{H}_i} V(H)$ and $C_i$, as desired.

        Since $|C_i|\leq |C_0|-i\leq n/200-i$, Stage 2 must end within at most $n/200$ steps.

        \item Stage 3: We have shown that we can admissibly remove vertices from $G$ such that the remaining vertex set is the disjoint union of $\bigcup_{J\in \mathcal{J}} V(J)$ and $D$, where $\mathcal{J}\subset \mathcal{H}$ has size at least $n/200$, and $D$ is a clique of size at most two consisting of vertices of odd degree.

        Let $L$ be a graph with vertex set $\mathcal{J}\cup\{D\}$ in which there is an edge between $J\in \mathcal{J}$ and $J'\in \mathcal{J}$ if and only if $V(J)$ is complete to $V(J')$, and there is an edge between $J\in \mathcal{J}$ and $D$ if and only if $V(J)$ is complete to $D$. By the degree bound (for $\overline{G}$), we have $\deg_L(D) \geq |\mathcal{J}|-\frac{2n}{10^5}> v(L)/2$, and $\deg_L(J) \geq |\mathcal{J}|-\frac{4n}{10^5}> v(L)/2$ for every $J\in \mathcal{J}$. Hence, by Dirac's theorem, there is a Hamilton cycle in $L$. Let us denote the vertices of this cycle as $D, J_1, J_2, \dots, J_{|\mathcal{J}|}$ in the natural order. 

        Since there is an edge between $J_1$ and $D$ in $L$, $V(J_1)$ is complete to $D$ and hence we can apply Lemma \ref{dense stage 3} to $V(J_1)$ and $D$ (in place of $R$ and $C$). Let $D'$ be the set of remaining vertices in $V(J_1)$. Since there is an edge between $J_2$ and $J_1$ in $L$, $V(J_2)$ is complete to $V(J_1)$ and hence also to $D'$, we can apply Lemma \ref{dense stage 3} to $V(J_2)$ and $D'$ (in place of $R$ and $C$). We repeat this with $J_3, \dots, J_{|\mathcal{J}|}$ in order. At the end of this procedure, we have at most two vertices left in the graph, and therefore the remaining graph is clearly even-decomposable. \qedhere
    \end{itemize}
\end{proof}

We now prove that, for a very wide range of values of $p$, it is likely that $G(n,p)$ contains at least $n/100$ vertex-disjoint induced copies of $P_3$. 

\begin{lem} \label{no notnecind copy of P_3}
    If a graph $G$ does not contain any (not necessarily induced) copies of $P_3$, then $e(G)\leq v(G)$. 
\end{lem}

\begin{proof}
    It suffices to prove that every connected component of $G$ is acyclic (a tree) or a copy of $C_3$. 

    Assume that some connected component $\mathcal{C}$ is not a tree, and not a copy of $C_3$. Let $C_k$ be a cycle of $k$ vertices in $\mathcal{C}$. If $k > 3$, then we can find a copy of $P_3$ in $C_k$. And if $k = 3$, let the cycle be $uvw$. Since $\mathcal{C}\neq \{u, v, w\}$ and $\mathcal{C}$ is connected, some vertex in $\{u, v, w\}$, say $u$, has an edge $ut$ to some $t\notin \{u, v, w\}$. Then $tuvw$ is a copy of $P_3$, a contradiction. 
\end{proof}

\begin{lem} \label{exact}
    If $p = \frac{100}{n}$ and $G\sim G(n, p)$, then with probability $1 - e^{-\Omega(n)}$, $G$ has at least $\frac{n}{10}$ vertex-disjoint induced copies of $P_3$. 
\end{lem}

\begin{proof}
    For a fixed subset $A$ of $k =  \frac{n}{2}$ vertices, applying the Chernoff bound to $e(G[A])\sim \mathrm{Bin}(\binom{k}{2}, p)$ gives $\mathbb{P}(e(G[A])\leq k) = O\left(e^{-n}\right)$. By Lemma \ref{no notnecind copy of P_3} applied to $G[A]$, the probability that $G[A]$ has no copies of $P_3$ is $O\left(e^{-n}\right)$. Using the union bound, the probability that there exists some $A\subseteq V(G)$ of size $k$ containing no copies of $P_3$ is at most $\binom{n}{k}O\left(e^{-n}\right) = e^{-\Omega(n)}$. Hence with probability $1 - e^{-\Omega(n)}$, we can keep taking vertex-disjoint copies of $P_3$, until there are less than $k$ vertices left, giving at least $\frac{n}{8}$ vertex-disjoint copies of $P_3$. 

    To find induced copies of $P_3$, we note that any non-induced copy of $P_3$ has at least four edges. The probability that a given set of four vertices induces at least four edges in $G(n,p)$ is $O(1/n^4)$. Hence, letting $\ell =  \frac{n}{8} - \frac{n}{10} $, the probability that there are at least $\ell$ disjoint sets of four vertices each inducing at least four edges is, by the union bound, at most $\left(\binom{n}{4}^\ell/\ell!\right)\cdot O(1/n^4)^\ell= O(1)^\ell/\ell! = e^{-\Omega(n)}$. Hence with probability $1 - e^{-\Omega(n)}$, there are at least $\frac{n}{8} - \ell = \frac{n}{10}$ vertex-disjoint induced copies of $P_3$ in $G$. 
\end{proof}

We extend this result to $p\in [\frac{100}{n}, 1 - \frac{100}{n}]$ by a sprinkling argument. 

\begin{lem} \label{many induced P_3}
    If $\frac{100}{n}\leq p\leq 1 - \frac{100}{n}$ and $G\sim G(n, p)$, then with probability $1 - e^{-\Omega(n)}$, $G$ has at least $\frac{n}{100}$ vertex-disjoint induced copies of $P_3$. 
\end{lem}

\begin{proof}
    Since $\overline{P_3} = P_3$, we may assume that $p\leq \frac{1}{2}$. 
    
    Note that $G$ can be generated as the union of $G_1$ and $G_2$, where $G_1\sim G(n,100/n)$ and $G_2\sim G(n,\tilde{p})$ are independent random graphs on the same vertex set, and $1-p = (1-100/n)(1 - \Tilde{p})$.
    Note that $0\leq \Tilde{p}\leq p\leq \frac{1}{2}$.
    
    By Lemma \ref{exact}, with probability $1 - e^{-\Omega(n)}$, we may find a collection $H_1, H_2, \dots, H_k$ of vertex-disjoint induced copies of $P_3$ in $G_1$ where $k = \frac{n}{10}$. Note that this event and the choice of $H_i$ are dependent only on $G_1$ and not on $G_2$. 
    
    From here onwards, we condition on $G_1$, and discuss probabilities with respect to the random choice of $G_2$. 

    For each $H_i$, with probability at least $\frac{1}{8}$, the three non-edges in $H_i$ are non-edges in $G_2$, and hence $H_i$ remains an induced copy of $P_3$ in $G$. 
    
    As $H_i$ are vertex-disjoint, and in particular edge-disjoint, these events are independent for each $H_i$. Hence the number of $H_i$s which remain induced copies of $P_3$ dominates the binomial distribution $\mathrm{Bin}(k, \frac{1}{8})$, which is at least $\frac{k}{10}$ with probability $1 - e^{-\Omega(n)}$ by the Chernoff bound.
\end{proof}

We are now ready to prove Theorem \ref{G(n, p)} in the dense regime.

\begin{thm} \label{p dense thm}
    If $1 - \frac{1}{10^6}\leq p\leq 1 - \frac{10^5}{n}$ and $G\sim G(n, p)$, then
    \begin{equation*}
        \mathbb{P}(G \text{ is even-non-decomposable}) = e^{-\Omega(n)}. 
    \end{equation*}
\end{thm}

\begin{proof}
    We lower bound the probability of three likely events, which together imply even decomposability by Lemma \ref{lem:deterministic dense}. 
    \begin{enumerate}
        \item For a fixed vertex $v$, $\deg_{\overline{G}}(v)\sim \mathrm{Bin}(n - 1, 1 - p)$ is dominated by $\mathrm{Bin}(n, \frac{1}{10^6})$. Applying the Chernoff bound to the latter gives that $\mathbb{P}(\deg_{\overline{G}}(v) > \frac{n}{10^5}) = e^{-\Omega(n)}$. By the union bound, with probability $1 - e^{-\Omega(n)}$, all vertices in the complement graph $\overline{G}$ have degree at most $\frac{n}{10^5}$. 
        \item We bound the clique number of $G$. The expected number of copies of $K_{\frac{n}{200}}$ in $G$ is $\binom{n}{n/200}p^{\binom{n/200}{2}}\leq (200e)^{n/200}e^{-\frac{10^5}{n}\binom{n/200}{2}}= e^{-\Omega(n)}$. Hence by Markov's inequality, with probability $1 - e^{-\Omega(n)}$, there are no copies of $K_{\frac{n}{200}}$ in $G$. 
        \item Lastly, by Lemma \ref{many induced P_3}, we have that with probability $1 - e^{-\Omega(n)}$, $G$ has at least $\frac{n}{100}$ vertex-disjoint induced copies of $P_3$.
    \end{enumerate}
    The theorem follows by applying Lemma \ref{lem:deterministic dense}.
\end{proof}

We now turn to the sparse regime. Using Lemmas \ref{sparse stage 2} and \ref{sparse stage 3} in place of Lemmas~\ref{dense stage 2} and \ref{dense stage 3}, one can prove the following analogue of Lemma \ref{lem:deterministic dense}.

\begin{lem} \label{lem:deterministic sparse}
    Let $G$ be a graph on $n$ vertices with an even number of edges. Assume that $G$ contains at least $\frac{n}{100}$ vertex-disjoint induced copies of $P_3$, $G$ has no copies of $K_{\frac{n}{200}}$, and all vertices have degree at most $\frac{n}{10^5}$ in $G$. Then $G$ is even-decomposable.
\end{lem}

Using this lemma and an argument nearly identical to the one in Theorem \ref{p dense thm}, we obtain the following result.

\begin{thm} \label{p sparse thm}
    If $\frac{10^5}{n}\leq p\leq \frac{1}{10^6}$ and $G\sim G(n, p)$, then
    \begin{equation*}
        \mathbb{P}(G\text{ is even-non-decomposable}) = e^{-\Omega(n)}. 
    \end{equation*}
\end{thm}

For very sparse random graphs, it is easy to prove that the probability of obtaining an even-non-decomposable graph is low.

\begin{prop} \label{p very sparse thm}
    If $p\leq\frac{10^5}{n}$ and $G\sim G(n, p)$, then
    \begin{equation*}
        \mathbb{P}(G \text{ is even-non-decomposable}) = O\left(\frac{1}{n^2}\right). 
    \end{equation*}
\end{prop}

\begin{proof}
    We bound the probability that $G$ contains $K_4$ as a subgraph, as all graphs with an even number of edges and without a copy of $K_4$ are even-decomposable (see Remark \ref{rem:k4free}). The expected number of copies of $K_4$ is $\binom{n}{4}p^6 = O\left(\frac{1}{n^2}\right)$, yielding the desired result by Markov's inequality. 
\end{proof}

We can summarise Theorem \ref{p average thm}, Theorem \ref{p dense thm}, Theorem \ref{p sparse thm} and Proposition \ref{p very sparse thm} in the following result, which proves Theorem \ref{G(n, p)} in a slightly stronger form.

\begin{thm}
    If $p\leq 1 - \frac{10^5}{n}$ and $G\sim G(n, p)$, then
    \begin{equation*}
        \mathbb{P}(G \text{ is even-non-decomposable}) = O\left(\frac{1}{n^2}\right). 
    \end{equation*}
\end{thm}

The following result shows that very dense random graphs are unlikely to be even-decomposable, and proves Proposition \ref{G(n, p) triv dir} in a slightly stronger form.

\begin{prop}
    If $p\geq 1 - \frac{10^{-1}}{n}$ and $G\sim G(n, p)$, then
    \begin{equation*}
        \mathbb{P}(G \text{ is even-decomposable}) = e^{-\Omega(n)}.
    \end{equation*}
\end{prop}

\begin{proof}
    Let $S$ be the set of isolated vertices in $\overline{G}$ (i.e., vertices in $G$ that are adjacent to all other vertices). For any $v\in S$, the only independent set in $G$ containing $v$ is $\{v\}$. Hence vertices in $S$ may only be removed by themselves (in a simple admissible removal step). We cannot remove two vertices from $S$ by themselves in a row, as all vertices in $S$ have the same degree, which changes parity after the removal of the first vertex. Hence a vertex outside of $S$ must be removed for every vertex we remove from $S$, so $G$ is not even-decomposable if $|S|> \frac{n}{2}+1$. 

    We apply the Chernoff bound to $\mathrm{Bin}(\binom{n}{2}, \frac{10^{-1}}{n})$ which dominates $e(\overline{G})\sim \mathrm{Bin}(\binom{n}{2}, 1-p)$, giving that with probability $1 - e^{-\Omega(n)}$, there are at most $n/5$ edges in $\overline{G}$, in which case $|S|> \frac{n}{2}+1$ and $G$ is not even-decomposable.
\end{proof}

\section{Even-degenerate graphs} \label{sec:degenerate}

In this section, we prove Theorem \ref{thm:degenerate}. Recall from Definition \ref{def:degenerate} that a graph $G$ is called even-degenerate if the vertices of $G$ may be ordered as $v_1, \cdots, v_n$ such that $v_i$ has even degree in $G[v_i, \cdots, v_n]$ for all $i\leq n - 2$. We call $G$ non-even-degenerate otherwise.

In other words, we may remove vertices with even degree (in the remaining graph) one by one, and $G$ is even-degenerate if we can remove all but at most two vertices. 

This operation of removing a vertex of even degree does not change the parity of the number of remaining edges. Hence if $e(G)$ is even, the number of edges between the (at most) two remaining vertices is also even. Therefore this remaining graph must be empty, and we may in fact legally remove all vertices. On the other hand, if $e(G)$ is odd, there must be an edge between the two remaining vertices and no further legal removals could be made.

The lower bound in Theorem \ref{thm:degenerate} is rather easy.

\begin{prop} \label{prop:deg lower}
    If $n\geq 4$ and $G\sim G(n,1/2)$, then with probability at least $(1/2)^{2n-3}$, $G$ is non-even-degenerate.
\end{prop}

\begin{proof}
    Assume first that $n$ is even. Fix $v\in V(G)$. Conditional on the graph $G-v$, there is a unique way to define the edges between $v$ and $V(G)\setminus \{v\}$ such that the degree of each vertex in $G$ is odd. Hence, with probability $(1/2)^{n-1}$, every vertex in $G$ has odd degree, in which case $G$ is non-even-degenerate.

    Assume now that $n$ is odd, and let $u\in V(G)$. With probability $(1/2)^{2n-3}$, $u$ is an isolated vertex and all other vertices have odd degree in $G$, in which case $G$ is non-even-degenerate.
\end{proof}

\begin{remark}
    There is a different construction also showing that with probability at least $(1/2)^{2n-2}$, $G$ is non-even-degenerate. Fix distinct vertices $u$ and $v$. Let $A$ be the event that $e(G)$ is even, $uv\in E(G)$ and $uw,vw\not \in E(G)$  holds for all $w\in V(G)\setminus \{u,v\}$. Since $n\geq 4$, we have $\mathbb{P}(A)=(1/2)^{2n-2}$. But observe that if $A$ holds, then $G$ is non-even-degenerate. Indeed, whichever of $u$ and $v$ comes first in the ordering, it must have degree one in the remaining graph. Hence, these two vertices must come last in a valid ordering. However, $e(G)$ is even and $uv$ is an edge, so a valid such ordering cannot exist by the observation made before Proposition \ref{prop:deg lower}.
\end{remark}

The rest of this section is devoted to proving the upper bound in Theorem \ref{thm:degenerate}. That is, we want to show that the probability that $G(n,1/2)$ is non-even-degenerate is at most $e^{-\Omega(n)}$.

\begin{df}
    Let $G\sim G(n, 1/2)$, and let $c_n = \mathbb{P}(G\text{ is non-even-degenerate})$. 
\end{df}

We shall first prove that $c_n = o(1)$ in Subsection \ref{o(1) bound}, which we will then promote to an exponential bound $c_n = e^{-\Omega(n)}$ in Subsection \ref{exp bound}. Crucially, the argument for the exponential bound relies on having shown that $c_n = o(1)$. We start by establishing some basic properties used in both arguments. 

\begin{lem} \label{uniform degree dist}
    For $G\sim G([n], 1/2)$, let $D\in\mathbb{F}_2^{[n]}$ be the $n$-tuple whose $i^\text{th}$ coordinate is the parity of $\deg(i)$. Then $D$ is distributed uniformly amongst $n$-tuples in $\mathbb{F}_2^{[n]}$ of sum $0$. 
\end{lem}

\begin{proof}
    For any $i\in [n]$, let $e_i\in\mathbb{F}_2^{[n]}$ be the basis vector having $1$ on the $i^{\text{th}}$ entry and $0$ elsewhere. 
    
    It suffices to show that $D$ and $D + e_i + e_j$ are identically distributed for all $i\neq j\in [n]$. To show this we consider the graph $G'$ obtained from $G$ by flipping $ij$. That is $ij$ is an edge in $G'$ iff it is not an edge in $G$, and all other edges are unchanged from $G$ to $G'$. We note that $G'\sim G(n, 1/2)$ is identically distributed to $G$ and $D + e_i + e_j$ is the $n$-tuple of the parities of the degrees in $G'$. Hence $D$ and $D + e_i + e_j$ are identically distributed. 
\end{proof}

\begin{lem} \label{stochastic forgetfulness}
    Let $G\sim G([n], 1/2)$, conditioned on $\deg(i)$ having parity $d_i$ for each $i = 1, \cdots, n$ (where $\sum_i d_i$ is even). Then for any $i\in V(G)$, we have $G - i\sim G(n - 1, 1/2)$. 
\end{lem}

\begin{proof}
    For each graph $H$ on vertex set $[n]\backslash\{i\}$, there is a unique way of extending it to a graph $G$ on vertex set $[n]$ satisfying the degree parity constraints. 
\end{proof}

Roughly speaking, the following proposition shows that the probability that $G(n,1/2)$ is not even-degenerate is decreasing in $n$.

\begin{prop} \label{decreasing c_n}
    For any $n$, we have $c_{n + 1}\leq (1 - 2^{-n})c_n + 2^{-n}$. 
\end{prop}

\begin{proof}
    Let $G\sim G([n + 1], 1/2)$. We query the parities of all degrees in $G$, on the basis of which we split into two cases. 

    \begin{itemize}
        \item Case 1: $\deg(i)$ is even for some vertex $i$. By Lemma \ref{uniform degree dist}, this case takes place with probability at least $1 - 2^{-n}$. 

        By Lemma \ref{stochastic forgetfulness}, $G - i\sim G(n, 1/2)$. Thus $G - i$ is non-even-degenerate with probability $c_n$. Note that if $G - i$ is even-degenerate, then so is $G$, as we may first remove $i$ and then follow through with the removal sequence for $G - i$. Hence in this case, $G$ is non-even-degenerate with probability at most $c_n$. 

        \item Case 2: all degrees are odd. By Lemma \ref{uniform degree dist}, this case takes place with probability at most $2^{-n}$. 
    \end{itemize}
    Combining the conditional probabilities, we obtain the desired bound on $c_{n + 1}$. 
\end{proof}

Both arguments (for $c_n = o(1)$ and $c_n = e^{-\Omega(n)}$) shall study how the likelihood of even-degeneracy of a graph $G$ changes if we condition on the outcome of an induced subgraph $G[S]$. The following definition sets up the necessary notation to study this question.

\begin{df} \label{def:linked}
    Given deterministic vertex sets $V, V'$ and $S\subseteq V\cap V'$, we call a pair of random graphs $(G, G')$ \emph{linked on $S$} if $G\sim G(V, 1/2), G'\sim G(V', 1/2)$, $G$ and $G'$ are conditioned on $G[S] = G'[S]$ and are independent otherwise. When $|S| = s, |V| = |V'| = m + s$, we denote by $b_{m, s}$ the probability that both $G$ and $G'$ are non-even-degenerate.
    
    We call $(G, G')$ \emph{parity-linked on $S$} if we further condition on $e(G)\equiv e(G') \mod 2$. When $|S| = s, |V| = |V'| = m + s$, we denote by $b_{m, s}^*$ the probability that both $G$ and $G'$ are non-even-degenerate.
\end{df}

\begin{example}
    If $V, V'\subseteq U$ and $G\sim G(U, 1/2)$, then $G[V], G[V']$ are linked on $V\cap V'$. 
\end{example}

As the notation indicates, the probabilities $b_{m, s}, b_{m, s}^*$ do not depend on the deterministic choices of the vertex sets $V, V', S$, only their sizes. In the next subsection, we study $b_{1,n-1}^*$ in order to relate $c_{n+1}$ to $c_n$. This will result in the bound $c_n=o(1)$. In Subsection \ref{exp bound} we study $b_{n-o(n),o(n)}$ in order to relate $c_{2n}$ to $c_n$. This will result in the bound $c_n = e^{-\Omega(n)}$. 

\subsection{The $o(1)$ bound} \label{o(1) bound}

In this subsection we prove that $c_n=o(1)$, i.e., that a random graph is almost surely even-degenerate. Although the proof is short, we think it is helpful to give a very brief overview.

We query the parities of the degrees of distinct vertices $u$ and $v$. The key observation is that if both $u$ and $v$ have even degree, then we have two different ways to certify that $G$ is even-degenerate: it suffices to show that either $G-u$ or $G-v$ is even-degenerate. The most involved part of the argument is to show that these two events (i.e. the event that $G-u$ is even-degenerate and the event that $G-v$ is even-degenerate) are not too closely correlated. This amounts to bounding the probability $b_{1,n-1}^*$ from Definition \ref{def:linked}.

\begin{lem} \label{s in terms of c}
    For any $n\geq 2$, we have $b_{1, n - 1}^*\leq c_n - c_n(1 - 2c_n)/n$. 
\end{lem}

\begin{proof}
    We inductively define random graphs $G^{(0)}, \cdots, G^{(n)}$ on vertex set $[n]$. 
    \begin{itemize}
        \item We take $G^{(0)}\sim G(n, 1/2)$. 
        \item For $m\in [n]$, we take $G^{(m)}\sim G(n, 1/2)$ conditioned on $G^{(m)} - m = G^{(m - 1)} - m$ and $e(G^{(m)})\equiv e(G^{(m - 1)}) \mod 2$. In other words, we form $G^{(m)}$ from $G^{(m - 1)}$ by uniformly resampling all edges incident to the vertex $m$, conditioned on leaving the parity of $\deg(m)$ unchanged. 
    \end{itemize}
    With this definition, $G^{(m - 1)}$ and $G^{(m)}$ are graphs on $[n]$ parity-linked on $[n]\backslash\{m\}$ for all $m\in [n]$; and $G^{(0)}, G^{(n)}$ are independent conditioned on $e(G^{(0)})\equiv e(G^{(m)}) \mod 2$. Hence, the probability that $G^{(m-1)}$ is non-even-degenerate and $G^{(m)}$ is even-degenerate is $c_n-b_{1,n-1}^*$ (as, by definition, the probability that both $G^{(m-1)}$ and $G^{(m)}$ are non-even-degenerate is $b_{1,n-1}^*$). On the other hand, we have
    \begin{align*}
        &\quad\mathbb{P}(G^{(0)} \text { is non-even-degenerate and } G^{(n)} \text{ is even-degenerate})\\
        &= c_n - \mathbb{P}(G^{(0)},  G^{(n)} \text { are both non-even-degenerate})\\
        &= c_n - \sum_{s\in \mathbb{F}_2} \mathbb{P}\left(G^{(0)},  G^{(n)} \text { are both non-even-degenerate and }e(G^{(0)})\equiv e(G^{(n)})\equiv s \mod 2\right)\\
        &= c_n - \sum_{s\in \mathbb{F}_2} \frac{1}{2}\cdot \mathbb{P}\left(G^{(0)},  G^{(n)} \text { are both non-even-degenerate } | \hspace{2mm} e(G^{(0)})\equiv e(G^{(n)})\equiv s \mod 2\right) \\
        &= c_n - \sum_{s\in \mathbb{F}_2} \frac{1}{2}\cdot \mathbb{P}\left(G^{(0)}\text { is non-even-degenerate } | \hspace{2mm} e(G^{(0)})\equiv s \mod 2\right) \\
        &\qquad\qquad\qquad\cdot\mathbb{P}\left(G^{(n)}\text { is non-even-degenerate } | \hspace{2mm} e(G^{(n)})\equiv s \mod 2\right)\\
        &= c_n - \sum_{s\in \mathbb{F}_2} 2\cdot\mathbb{P}\left(G^{(0)}\text { is non-even-degenerate and } e(G^{(0)})\equiv s \mod 2\right) \\
        &\qquad\qquad\qquad\cdot\mathbb{P}\left(G^{(n)}\text { is non-even-degenerate and }e(G^{(n)})\equiv s \mod 2\right)\\
        &\geq c_n - \sum_{s\in \mathbb{F}_2} 2\cdot\mathbb{P}\left(G^{(0)}\text { is non-even-degenerate and } e(G^{(0)})\equiv s \mod 2\right)\cdot c_n \\
        &= c_n - 2c_n^2. 
    \end{align*}

    Whenever $G^{(0)}$ is non-even-degenerate whilst $G^{(n)}$ is even-degenerate, there must be some $m \in [n]$ for which $G^{(m-1)}$ is non-even-degenerate and $G^{(m)}$ is even-degenerate. Taking the union bound, we have
    \begin{align*}
        c_n - 2c_n^2
        &=\mathbb{P}(G^{(0)} \text { is non-even-degenerate and } G^{(n)} \text{ is even-degenerate}) \\
        &\leq \sum_{m\in [n]} \mathbb{P}(G^{(m-1)} \text { is non-even-degenerate and } G^{(m)} \text{ is even-degenerate})=n(c_n-b_{1,n-1}^*).
    \end{align*}
    Rearranging this inequality, we obtain $b_{1, n - 1}^*\leq c_n - c_n(1 - 2c_n)/n$. 
\end{proof}

\begin{lem} \label{o(1) bound lem}
    For any $n\geq 2$, we have $c_{n + 1}\leq \frac{1}{4}b_{1, n - 1}^* + \left(\frac{3}{4} - 2^{-n}\right)c_n + 2^{-n}$. 
\end{lem}

\begin{proof}
    Let $G\sim G([n + 1], 1/2)$. We query the parities of $\deg(n)$ and $\deg(n + 1)$, based on which we split into two cases. 

    \begin{itemize}
        \item Case 1: $\deg(n), \deg(n + 1)$ are both even. By Lemma \ref{uniform degree dist}, this case takes place with probability $1/4$. 

        If $G - n$ (respectively, $G-(n+1)$) is even-degenerate, then so is $G$ since we may remove $n$ (resp. $n+1$) first and then decompose $G-n$ (resp. $G-(n+1)$). Hence, if $G$ is non-even-degenerate, then both $G-n$ and $G-(n+1)$ must be non-even-degenerate. Note that conditional on the event that both $\deg(n)$ and $\deg(n+1)$ are even, the graphs $G - n$ and $G - (n + 1)$ are parity-linked on $[n - 1]$. Hence, in this case, $G$ is non-even-degenerate with probability at most $b_{1, n - 1}^*$. 

        \item Case 2: $\deg(n), \deg(n + 1)$ are not both even. 

        We query the parities of the degrees of all other vertices in $G$, on the basis of which we further split into two subcases. 

        \begin{itemize}
            \item Subcase 1: $\deg(i)$ is even for some vertex $i$. By Lemma \ref{uniform degree dist}, this case takes place with overall probability at least $3/4 - 2^{-n}$. 

            Remove $i$ from $G$, and by Lemma \ref{stochastic forgetfulness}, $G - i\sim G(n, 1/2)$. Thus $G - i$ is non-even-degenerate with probability $c_n$. 

            Hence in this subcase, $G$ is non-even-degenerate with probability at most $c_n$. 
            \item Subcase 2: all degrees are odd. By Lemma \ref{uniform degree dist}, this case takes place with overall probability at most $2^{-n}$. 
        \end{itemize}
    \end{itemize}
    The desired upper bound on $c_{n + 1}$ now follows from combining the conditional probabilities above. 
\end{proof}

Combining Lemmas \ref{o(1) bound lem} and \ref{s in terms of c}, we obtain the following recursive bound on $c_n$.

\begin{cor} \label{cor:o(1) recursion}
    For all $n\geq 2$, we have $c_{n+1}\leq c_n-\frac{c_n(1-2c_n)}{4n}-2^{-n}c_n+2^{-n}$.
\end{cor}

We are now ready to prove the main result of this subsection.

\begin{prop} \label{o(1) prop}
    We have $c_n = o(1)$. 
\end{prop}

\begin{proof}
    We have $c_1 = c_2 = c_3 = 0$, and for $n\geq 4$, by Proposition \ref{decreasing c_n}, $c_n\leq c_3 + 2^{-3} + \cdots + 2^{-(n - 1)}\leq 1/4$. Hence $1-2c_n\geq 1/2$ always. 
    
    Thus by Corollary \ref{cor:o(1) recursion}, for any $n\geq 2$, we have $c_{n + 1}\leq (1 - 1/(8n))c_n + 2^{-n}$. Thus $c_{n + 1} + 10\cdot 2^{-(n + 1)}\leq (1 - 1/(8n))c_n + 6\cdot 2^{-n}\leq (1 - 1/(8n))\left(c_n + 10\cdot 2^{-n}\right)\leq (1 + 1/n)^{-1/8}\left(c_n + 10\cdot 2^{-n}\right)$. In other words, for any $n\geq 2$, we have $(n + 1)^{1/8}\left(c_{n + 1} + 10\cdot 2^{-(n + 1)}\right)\leq n^{1/8}\left(c_n + 10\cdot 2^{-n}\right)$. 

    Thus for any $n\geq 2$, we have $n^{1/8}\left(c_n + 10\cdot 2^{-n}\right)\leq 2^{1/8}\left(c_2 + 10\cdot 2^{-2}\right)\leq 10$. Thus $c_n \leq 10n^{-1/8} = o(1)$. 
\end{proof}

\subsection{The exponential bound} \label{exp bound}

In this subsection we prove that $c_n=e^{-\Omega(n)}$. This is achieved via establishing a recursive inequality, similar to Corollary \ref{cor:o(1) recursion}, but this time relating $c_{2n}$ to $c_n$.

For technical reasons, we introduce the following function.

\begin{df}
    Let $f_n = \max\left(2^{-n/6}, \max_{n'\leq n}c_{n'}2^{(n' - n)/2}\right)$.
\end{df}

Note that trivially $f_n\geq c_n$, so it will be sufficient to upper bound $f_n$. In fact, the reader is encouraged to think of $f_n$ and $c_n$ as the same function.

The key recursive inequality is as follows.

\begin{prop} \label{key}
    For any positive integers $a\leq t\leq n$ with $a < n/2$, we have 
    \begin{equation*}
        c_{2n} \leq 256\cdot t^{4a}2^{3a}f_n^2 + 14\cdot 2^{-t/2} + 8\cdot 2^{-a^2}. 
    \end{equation*}
\end{prop}

Once Proposition \ref{key} is established, one can easily translate this inequality into a bound for $f_{2n}$ in terms on $f_n$. Solving that recursion is tedious, but straightforward, and will be done in Section~\ref{recursion solving}. We now turn to the more interesting task of proving Proposition \ref{key}. Similarly to the argument in the previous subsection (see Lemma \ref{s in terms of c}), it is crucial to upper bound the function $b_{m,s}$ from Definition \ref{def:linked}, and here we need a bound for general values of $m$ and $s$. 

\begin{lem} \label{link cor}
    For any integers $m\geq 1$ and $s\geq 0$, we have $b_{m, s}\leq 16\cdot 2^{3s/2} f_{m + s}^2$. 
\end{lem}

Recall that $b_{m,s}$ is the probability that both $G$ and $G'$ are non-even-degenerate where $G$ and $G'$ are (random) graphs on $m+s$ vertices, linked on a set of size $s$. Roughly speaking, Lemma \ref{link cor} says that when $s$ is small, then this probability is not much bigger than the probability that two independent random graphs are both non-even-degenerate (indeed, the latter probability is $c_{m+s}^2$, which we think of as the same as $f_{m+s}^2$).

We will prove Lemma \ref{link cor} in Section \ref{sec:linked}. Now we proceed with the proof of Proposition \ref{key}. The following definition plays a key role in our proof.

\begin{df}
    For any positive integers $a, t\leq n$, we call a subset $C\subseteq [2n]$ \emph{$(t, a)$-initial} if the following all hold. 
    \begin{enumerate}
        \item $|C| = n$, 
        \item $[n - t]\subseteq C\subseteq [n + t]$, 
        \item $|[n]\backslash C|\leq a$. 
    \end{enumerate}
    We call a subset $C\subseteq [2n]$ \emph{$(t, a)$-terminal} if the reversed set $\{2n + 1 - c:c\in C\}$ is \emph{$(t, a)$-initial}. 
\end{df}

\begin{lem} \label{key stage 1}
    Let $G\sim G([2n], 1/2)$. Let $a, t\leq n$ be positive integers. Define vertices $v_1,v_2,\dots,v_n$ recursively as follows. Having defined $v_1,\dots,v_i$, we let $v_{i+1}$ be the vertex of smallest label in $G-\{v_1,\dots,v_i\}$ amongst vertices of even degree in $G-\{v_1,\dots,v_i\}$. If no such vertex exists, then let $v_{i+1}$ be a uniformly random vertex in $G-\{v_1,\dots,v_i\}$. Let $C=\{v_1,\dots,v_n\}$ be the set of removed vertices. Denote by $F$ the event that each $v_{i+1}$ has even degree in $G-\{v_1,\dots,v_i\}$ and $C$ is $(t, a)$-initial. 
        
    Then $F$ holds with probability at least $1 - 7\cdot 2^{-t/2} - 4\cdot 2^{-a^2}$. 
\end{lem}

\begin{proof}
    Note that $F$ holds if the following all hold. 
    \begin{enumerate}
        \item  each $v_{i+1}$ has even degree in $G-\{v_1,\dots,v_i\}$, 
        \item $[n - t]\subseteq C\subseteq [n + t]$, 
        \item $|[n]\backslash C|\leq a$. 
    \end{enumerate}
    We shall upper bound the failure probability of each. We will write $C_i=\{v_1,\dots,v_i\}$ for the set of vertices removed in the first $i$ steps and we write $G_i=G-C_i$ for the remaining graph after $i$ steps. Note that by Lemma \ref{stochastic forgetfulness}, conditional on the vertices $v_1,v_2,\dots,v_i$, the graph $G_i$ is uniformly random amongst graphs on vertex set $[2n]\setminus C_i$.
    \begin{enumerate}
        \item Since $G_i$ is uniformly random, by Lemma \ref{uniform degree dist}, the probability that all vertices of $G_i$ have odd degree is at most $2^{-(2n-i-1)}$. Hence, by the union bound, the probability that there is some $0\leq i\leq n-1$ such that $v_{i+1}$ has odd degree in $G-\{v_1,\dots,v_i\}$ is bounded above by $\sum_{i=0}^{n-1} 2^{-(2n-i-1)}\leq 2^{- n + 1}$. 
    
        \item We now upper bound the probability that $C\not \subseteq [n + t]$. If $v_i > n + t$ for some positive integer $i$, then all vertices in $[n + t]\backslash C_{i - 1}$ have odd degrees in $G_{i - 1}$. Hence $\mathbb{P}(v_i > n + t)\leq 2^{-(n + t - i + 1)}$. Summing over all $i\in [n]$ and taking the union bound, we have $\mathbb{P}(C\not \subseteq [n + t])\leq 2^{-t}$. 
    
        We then upper bound the probability that $[n - t]\not \subseteq C$. Let $M_i = |[n - t]\backslash C_i|$. In this language $[n - t]\not \subseteq C$ is equivalent to $M_n > 0$. Since for each $i$, conditional on the vertices $v_1,\dots,v_i$, the graph $G_i$ is uniformly random amongst graphs on vertex set $[2n]\setminus C_i$, it follows from Lemma~\ref{uniform degree dist} that $M_0, \dots, M_n$ form a homogeneous Markov chain with initial state $M_0 = n - t$ and transition probabilities $\mathbb{P}(M_{i + 1} = k|M_i = k) = 2^{-k}$, $\mathbb{P}(M_{i + 1} = k - 1|M_i = k) = 1 - 2^{-k}$. Let $T_k$ be the hitting time of $k$ (were we to extend the Markov chain indefinitely). Clearly, $T_{k - 1} - T_k\sim \textrm{Geo}(1 - 2^{-k})$ are mutually independent (here and below, geometric distributions are supported on $\mathbb{Z}^+$). Hence 
        \begin{align*}
            \mathbb{E}(2^{T_0/2}) &= \mathbb{E}(2^{(T_0 - T_{n - t})/2}) = \prod_{k = 1}^{n - t}\mathbb{E}(2^{(T_{k - 1} - T_k)/2})= \prod_{k = 1}^{n - t}\left(2^{1/2}\cdot\frac{1 - 2^{-k}}{1 - 2^{1/2 - k}}\right)\\
            &= 2^{(n - t)/2}\cdot\frac{\prod_{k = 1}^{n - t}\left(1 - 2^{-k}\right)}{\prod_{k = 1}^{n - t}\left(1 - 2^{1/2-k}\right)}\leq 2^{(n - t)/2}\cdot\frac{1}{1 - 2^{-1/2}}\leq 4\cdot 2^{(n - t)/2}, 
        \end{align*}
        and therefore $\mathbb{P}([n - t]\not \subseteq C) = \mathbb{P}(M_n > 0) = \mathbb{P}(T_0 > n) \leq 4\cdot 2^{-t/2}$ by Markov's inequality.
    
        \item We bound the probability that $|[n]\backslash C| > a$. We similarly construct $M'_i = |[n]\backslash C_i|$, and $M'_0, \dots, M'_n$ form a homogeneous Markov chain with initial state $M'_0 = n$ and the same transition probabilities as $M_0, \dots, M_n$. In this setting, $|[n]\backslash C| > a$ is equivalent to $M'_n > a$. Let $T'_k$ be the hitting time of $k$ (were we to extend the Markov chain indefinitely). Again, $T'_{k - 1} - T'_k\sim \textrm{Geo}(1 - 2^{-k})$ are mutually independent. Hence
        \begin{align*}
            \mathbb{E}(2^{a T'_a}) &= \mathbb{E}(2^{a(T'_a - T'_n)}) = \prod_{k = a + 1}^n\mathbb{E}(2^{a(T'_{k - 1} - T'_k)}) = \prod_{k = a + 1}^n\left(2^a\cdot\frac{1 - 2^{-k}}{1 - 2^{a - k}}\right)\\
            &= 2^{a(n - a)}\cdot\frac{\prod_{k = a + 1}^n\left(1 - 2^{-k}\right)}{\prod_{k = a + 1}^n\left(1 - 2^{a - k}\right)}\leq 2^{a(n - a)}\cdot\frac{1}{\prod_{i = 1}^{\infty}\left(1 - 2^{-i}\right)}\\
            &\leq 2^{a(n - a)}\cdot\frac{1}{\left(1 - 2^{-1}\right)\left(1 - \sum_{i = 2}^{\infty} 2^{-i}\right)} = 4\cdot 2^{a(n - a)}, 
        \end{align*}
        and therefore $\mathbb{P}(|[n]\backslash C| > a) = \mathbb{P}(M'_n > a) = \mathbb{P}(T'_a > n) \leq 4\cdot 2^{-a(a + 1)}\leq 4\cdot 2^{-a^2}$ by Markov's inequality. 
    \end{enumerate}
    Hence $F$ holds with probability at least $1 - 2^{-n + 1} - 2^{-t} - 4\cdot 2^{-t/2} - 4\cdot 2^{-a^2}\geq 1 - 7\cdot 2^{-t/2} - 4\cdot 2^{-a^2}$. 
\end{proof}

We are now ready to prove Proposition \ref{key}.

\begin{proof}[Proof of Proposition \ref{key}]
    Let $G\sim G([2n], 1/2)$. 

    Consider any fixed $(t, a)$-initial subset $D$ and $(t, a)$-terminal subset $\Tilde{D}$. The induced subgraphs $G - D, G - \Tilde{D}$ are linked on $([2n]\backslash D)\cap ([2n]\backslash \Tilde{D})$. Let $s = |([2n]\backslash D)\cap ([2n]\backslash \Tilde{D})|$. Since $[n]$ and $\{n + 1, \dots, 2n\}$ partition $[2n]$, we have $$s\leq |[n]\cap ([2n]\backslash D)| + |\{n + 1, \dots, 2n\}\cap ([2n]\backslash \Tilde{D})| = |[n]\backslash D| + |\{n + 1, \dots, 2n\}\backslash \Tilde{D}|\leq 2a.$$
    Hence, by Lemma \ref{link cor} we have $$\mathbb{P}(G - D, G - \Tilde{D}\text{ are both non-even-degenerate}) = b_{n - s, s}\leq 16\cdot 2^{3s/2}f_n^2\leq 16\cdot 2^{3a}f_n^2.$$

    A $(t, a)$-initial subset $D$ is determined by $[n]\backslash D$ and $\{n + 1, \dots, 2n\}\cap D$. As $|[n]\backslash D|\leq a$ and $[n]\backslash D\subseteq \{n - t + 1, \dots, n\}$, there are at most ${t \choose a} + \dots + {t \choose 0}\leq t^a + 1\leq 2\cdot t^a$ choices of $[n]\backslash D$. Similarly, as $|\{n + 1, \dots, 2n\}\cap D| = |[n]\backslash D| \leq a$ and $\{n + 1, \dots, 2n\}\cap D\subseteq \{n + 1, \dots, n + t\}$, there are at most ${t \choose a} + \dots + {t \choose 0}\leq t^a + 1\leq 2\cdot t^a$ choices of $\{n + 1, \dots, 2n\}\cap D$. Hence, there are at most $4\cdot t^{2a}$ $(t, a)$-initial subsets, and likewise there are at most $4\cdot t^{2a}$ $(t, a)$-terminal subsets. Therefore, by the union bound, 
    \begin{equation*}
        \mathbb{P}\left(\exists (t, a)\text{-initial }D, (t, a)\text{-terminal }\Tilde{D}\text{ s.t. }G - D, G - \Tilde{D}\text{ are both non-even-degenerate}\right)\leq 256\cdot t^{4a}2^{3a}f_n^2. 
    \end{equation*}

    By Lemma \ref{key stage 1}, with probability at least $1 - 7\cdot 2^{-t/2} - 4\cdot 2^{-a^2}$, we may legally remove a $(t, a)$-initial subset of vertices. Likewise, with probability at least $1 - 7\cdot 2^{-t/2} - 4\cdot 2^{-a^2}$, we may legally remove a $(t, a)$-terminal subset of vertices. Let $E$ be the event that we succeed in both, i.e., that there exist a $(t, a)$-initial subset $C$ and a $(t, a)$-terminal subset $\Tilde{C}$ which may (separately) both be legally removed. Then $E$ holds with probability at least $1 - 14\cdot 2^{-t/2} - 8\cdot 2^{-a^2}$. If $G$ is non-even-degenerate and $E$ holds, then both $G - C$ and $G - \Tilde{C}$ are non-even-degenerate.
    
    Hence the probability that $G$ is non-even-degenerate and $E$ holds is at most $256\cdot t^{4a}2^{3a}f_n^2$. Therefore $c_{2n} = \mathbb{P}(G \text{ is non-even-degenerate}) \leq 256\cdot t^{4a}2^{3a}f_n^2 + 14\cdot 2^{-t/2} + 8\cdot 2^{-a^2}$, as needed.
\end{proof}

\subsubsection{Proof of Lemma \ref{link cor}} \label{sec:linked}

In this subsection we prove Lemma \ref{link cor}. We will study how fixing the edges between a set of vertices in our random graph changes the probability that the graph is even-degenerate.

\begin{df}
    For a proper subset of vertices $S$, and a graph $L$ on vertex set $S$, let $G(n, S, L)$ be the distribution of $G$ where $G\sim G(n, 1/2)$ conditioned on $G[S] = L$. Let $c_{n, S, L}$ be the probability that $G$ is non-even-degenerate. 
\end{df}
\begin{df}    
    For $m\geq 1$ and $s\geq 0$, let $c_{m, s}$ be the maximum of $c_{m + s, S, L}$ amongst $S$ with $|S| = s$. Here $m$ is the number of ``unconditioned" vertices. 
\end{df}

To avoid carrying a small additive term throughout the following argument, we shift this term into $c_{m, s}$. 

\begin{df}
    For $m\geq 1$ and $s\geq 0$, let $g_{m, s} = c_{m, s} + 2^{-(m - 1)}$. 
\end{df}

We shall need the following generalisations of Lemma \ref{stochastic forgetfulness}, whose proofs are nearly identical to that of Lemma \ref{stochastic forgetfulness}. 

\begin{lem} \label{conditioned stochastic forgetfulness 1}
    Let $G\sim G([n], S, L)$ conditioned on $\deg(i)$ having parity $d_i$ for each $i \in [n]$ (where $\sum_i d_i$ is even). Then for any $i\notin S$, we have $G - i\sim G(n - 1, S, L)$. 
\end{lem}

\begin{lem} \label{conditioned stochastic forgetfulness 2}
    Let $G\sim G([n], S, L)$ conditioned on $\deg(i)$ having parity $d_i$ for each $i \in [n]$ (where $\sum_i d_i$ is even). Then for any $i\in S$, we have $G - i\sim G(n - 1, S\backslash\{i\}, L - i)$ conditioned upon $\deg_{G - i}(i')$ having parity $d_{i'} - L_{i'i}$ for all $i'\in S\backslash\{i\}$. 
\end{lem}

\begin{lem} \label{conditioned recurrence}
    For any $m\geq 2$ and $s\geq 1$, we have $g_{m, s}\leq 2^{-s}g_{m - 1, s} + \left(1 - 2^{-s}\right)g_{m - 1, s - 1}$. 
\end{lem}

\begin{proof}
    Let $n = m + s$. Take $G\sim G(n, S, L)$ where $|S| = s$, it suffices to bound from above the probability that $G$ is non-even-degenerate. 
    
    We query the parities of all degrees in $G$, based on which we distinguish between three cases. 
    
    \begin{itemize}
        \item Case 1: $\deg(i)$ is even for some $i\in S$. By Lemma \ref{uniform degree dist}, this case takes place with probability $1 - 2^{-s}$. 
        
        Remove $i$ and note that $G - i\sim G(n - 1, S\backslash\{i\}, L - i)$ with conditions on the parities of the degrees of vertices in $S\backslash\{i\}$ given by Lemma \ref{conditioned stochastic forgetfulness 2}. 

        Query the parities of all degrees in $G - i$. By Lemma~\ref{uniform degree dist}, with probability at least $1 - 2^{-(m - 1)}$, some $j\notin S$ has even degree in $G-i$. Remove $j$ from $G - i$. By Lemma \ref{conditioned stochastic forgetfulness 1}, $G-i-j$ has distribution $G(n - 2, S\backslash\{i\}, L - i)$. 
        
        Hence in this case, $G$ is non-even-degenerate with probability at most $\left(1 - 2^{-(m - 1)}\right)c_{m - 1, s - 1} + 2^{-(m - 1)}$. 
        
        \item Case 2: $\deg(i)$ is odd for all $i\in S$, but $\deg(j)$ is even for some $j\notin S$. By Lemma \ref{uniform degree dist}, this case takes place with probability at least $2^{-s} - 2^{-(n - 1)}$. 
        
        Remove $j$ and note that $G - j\sim G(n - 1, S, L)$. Hence in this case, $G$ is non-even-degenerate with probability at most $c_{m - 1, s}$. 
        \item Case 3: all degrees are odd in $G$. By Lemma \ref{uniform degree dist}, this case takes place with probability at most $2^{-(n - 1)}$. 
    \end{itemize}
    Combining the conditional probabilities in each case, we have 
    \begin{align*}
        c_{n, S, L} \leq & \left(1 - 2^{-s}\right)\left(\left(1 - 2^{-(m - 1)}\right)c_{m - 1, s - 1} + 2^{-(m - 1)}\right)\\
        &+ \left(2^{-s} - 2^{-(n - 1)}\right)c_{m - 1, s} + 2^{-(n - 1)} \\
        = & \left(1 - 2^{-(m - 1)}\right)\left(2^{-s}c_{m - 1, s} + \left(1 - 2^{-s}\right)c_{m - 1, s - 1}\right) + 2^{-(m - 1)}\\
        \leq & 2^{-s}c_{m - 1, s} + \left(1 - 2^{-s}\right)c_{m - 1, s - 1} + 2^{-(m - 1)}, 
    \end{align*}
    for all $S$ with $|S| = s$ and graph $L$ on $S$. Hence $c_{m, s}\leq 2^{-s}c_{m - 1, s} + \left(1 - 2^{-s}\right)c_{m - 1, s - 1} + 2^{-(m - 1)}$, and the desired conclusion follows. 
\end{proof}

Taking $g_{m, s} = 2$ when $m\leq 0$, the inequality in Lemma \ref{conditioned recurrence} holds for all $m\in \mathbb{Z}$ and $s\geq 1$. Under this convention, iterating Lemma \ref{conditioned recurrence} gives the following result.

\begin{cor} \label{cor 1}
    For any $m\in\mathbb{Z}, s\geq 1$, we have $g_{m, s}\leq\mathbb{E}(g_{m - X_s, s - 1})$ where $X_s\sim Geo(1 - 2^{-s})$.
\end{cor}

\begin{proof}
    We prove by induction on $k\in\mathbb{Z}_{\geq 0}$ that
    \begin{equation*}
        g_{m, s}\leq \sum_{i = 1}^k \left(\left(1 - 2^{-s}\right)2^{-(i - 1)s}\cdot g_{m - i, s - 1}\right) + 2^{-ks}\cdot g_{m - k, s}. 
    \end{equation*}
    The base case $k = 0$ is immediate. For the induction step, we assume that the inequality holds for some $k\in\mathbb{Z}_{\geq 0}$. Applying (the above generalisation of) Lemma \ref{conditioned recurrence} to $g_{m - k, s}$, we have
    \begin{align*}
        g_{m, s}&\leq \sum_{i = 1}^k \left(\left(1 - 2^{-s}\right)2^{-(i - 1)s}\cdot g_{m - i, s - 1}\right) + 2^{-ks}\cdot g_{m - k, s}\\
        &\leq \sum_{i = 1}^k \left(\left(1 - 2^{-s}\right)2^{-(i - 1)s}\cdot g_{m - i, s - 1}\right) + 2^{-ks}\cdot \left(2^{-s}g_{m - k - 1, s} + \left(1 - 2^{-s}\right)g_{m - k - 1, s - 1}\right)\\
        &= \sum_{i = 1}^{k + 1} \left(\left(1 - 2^{-s}\right)2^{-(i - 1)s}\cdot g_{m - i, s - 1}\right) + 2^{-(k + 1)s}\cdot g_{m - k - 1, s}, 
    \end{align*}
    as needed, completing the induction. Noting that $g_{m, s}$ is uniformly bounded above by $2$, we may take the limit $k\rightarrow\infty$, giving
    \begin{align*}
        g_{m, s}&\leq \sum_{i = 1}^\infty \left(\left(1 - 2^{-s}\right)2^{-(i - 1)s}\cdot g_{m - i, s - 1}\right)\\
        &= \sum_{i = 1}^\infty \left(\mathbb{P}(X_s = i)\cdot g_{m - i, s - 1}\right)\\
        &=\mathbb{E}(g_{m - X_s, s - 1}),
    \end{align*}
    for $X_s\sim Geo(1 - 2^{-s})$. 
\end{proof}

Further iterating, we obtain the following result bounding the $g_{m, s}$'s in terms of the $c_n$'s. 

\begin{cor} \label{g** in terms of c*}
    For any integers $m\geq 1$ and $s\geq 0$, we have $$g_{m, s}\leq \mathbb{E}\left(\mathbbm{1}_{X< m}\left(c_{m - X} + 2^{-(m - X - 1)}\right)\right) + 2\mathbb{P}(X\geq m),$$ where $X = \sum_{t = 1}^s X_t$ with $X_t \sim \textrm{Geo}(1 - 2^{-t})$ independent. 
\end{cor}

\begin{proof}
    By Corollary \ref{cor 1}, we have
    \begin{align*}
        g_{m, s}&\leq\mathbb{E}(g_{m - X_s, s - 1})\\
        &\leq \mathbb{E}(g_{m - X_s - X_{s - 1}, s - 2})\\
        &\leq \cdots\\
        &\leq \mathbb{E}(g_{m - X_s - X_{s - 1} - \cdots - X_1, 0})\\
        &= \mathbb{E}(g_{m - X, 0}). 
    \end{align*}
    As $g_{n, 0} = \begin{cases}
        c_n + 2^{-(n - 1)} & \text{ if } n > 0, \\
        2 & \text{ otherwise,}
    \end{cases}$, we have
    \begin{align*}
        g_{m, s}&\leq \mathbb{E}(g_{m - X, 0})\\
        & = \mathbb{E}\left(\mathbbm{1}_{X < m}\left(c_{m - X} + 2^{-(m - X - 1)}\right)\right) + 2\mathbb{P}(X\geq m),
    \end{align*}
    as needed. 
\end{proof}

\begin{lem} \label{g** in terms of f*}
    For any integers $m\geq 1$ and $s\geq 0$, we have $g_{m, s}\leq 16\cdot 2^s f_m$. 
\end{lem}

\begin{proof}
    By Corollary \ref{g** in terms of c*}, we have $g_{m, s}\leq\mathbb{E}\left(\mathbbm{1}_{X< m}\left(c_{m - X} + 2^{-(m - X - 1)}\right)\right) + 2\mathbb{P}(X\geq m)$ where $X = \sum_{t = 1}^s X_t$ with $X_t \sim \textrm{Geo}(1 - 2^{-t})$ independent. By the definition of $f_m$, we have $c_{m - X}\leq 2^{X/2}f_m$ whenever $X<m$. Therefore
    \begin{align*}
        g_{m, s}&\leq\mathbb{E}\left(\mathbbm{1}_{X < m}\left(c_{m - X} + 2^{-(m - X - 1)}\right)\right) + 2\mathbb{P}(X\geq m)\\
        &\leq f_m\mathbb{E}\left(2^{X/2}\right) + 2^{-(m - 1)}\mathbb{E}\left(\mathbbm{1}_{X < m}2^X\right) + 2\mathbb{P}(X\geq m). 
    \end{align*}
    Let $Y = \sum_{t = 2}^s X_t$. We calculate the moments $\mathbb{E}(2^Y) = \prod_{t = 2}^s \mathbb{E}(2^{X_t}) = \prod_{t = 2}^s \left(2\cdot\frac{1 - 2^{-t}}{1 - 2^{1 - t}}\right)\leq 2^s$ and $\mathbb{E}(2^{X/2}) = \prod_{t = 1}^s \mathbb{E}(2^{X_t/2}) = \prod_{t = 1}^s \left(2^{1/2}\cdot\frac{1 - 2^{-t}}{1 - 2^{1/2 - t}}\right)\leq \left(1 - 2^{-1/2}\right)^{-1}2^{s/2}\leq 4\cdot 2^{s/2}$. Markov's inequality now gives $\mathbb{P}(X \geq m)\leq 4\cdot 2^{(s - m)/2}$. Hence
    \begin{align*}
        g_{m, s}&\leq f_m\mathbb{E}\left(2^{X/2}\right) + 2^{-(m - 1)}\mathbb{E}\left(\mathbbm{1}_{X < m}2^X\right) + 2\mathbb{P}(X>m)\\
        &\leq 4\cdot 2^{s/2}f_m + 2^{-(m - 1)}\mathbb{E}\left(\mathbbm{1}_{X_1 < m}2^{X_1}\right)\mathbb{E}(2^Y) + 8\cdot 2^{(s - m)/2}\\
        &\leq 4\cdot 2^{s/2}f_m + (m-1)2^{s - m + 1} + 8\cdot 2^{(s - m)/2}\\
        &\leq 16\cdot 2^s f_m, 
    \end{align*}
    where the last inequality uses $f_m\geq 2^{-m/6}$.
\end{proof}

\begin{lem} \label{decreasing f_n}
    For any $n\geq 1$, we have $2^{-1/2}f_n\leq f_{n+1}\leq f_n + 2^{-n}$. 
\end{lem}

\begin{proof}
    For the lower bound we have $f_{n + 1}\geq \max_{n'\leq n + 1}c_{n'}2^{(n' - n - 1)/2}\geq 2^{-1/2}\max_{n'\leq n}c_{n'}2^{(n' - n)/2}$, and $f_{n + 1}\geq 2^{-(n + 1)/6}\geq 2^{-1/2}2^{-n/6}$. Hence $f_{n + 1}\geq 2^{-1/2}\max\left(2^{-n/6}, \max_{n'\leq n}c_{n'}2^{(n' - n)/2}\right) = 2^{-1/2}f_n$. 

    For the upper bound, we have 
    \begin{align*}
        f_{n + 1} &= \max\left(2^{-(n + 1)/6}, \max_{n'\leq n + 1}c_{n'}2^{(n' - n - 1)/2}\right)\\
        &\leq \max\left(2^{-n/6}, \max_{n'\leq n}c_{n'}2^{(n' - n)/2}, c_{n + 1}\right)\\
        &=\max\left(f_n, c_{n + 1}\right), 
    \end{align*}
    Hence it suffices to show $c_{n + 1}\leq f_n + 2^{-n}$. This follows from Proposition \ref{decreasing c_n}, as $c_{n + 1}\leq \left(1 - 2^{-n}\right)c_n + 2^{-n}\leq f_n + 2^{-n}$.
\end{proof}

We are now ready to prove Lemma \ref{link cor}.

\begin{proof}[Proof of Lemma \ref{link cor}]
    Let $G, G'$ be graphs on vertex sets $V, V'$, respectively, and linked on $S$ where $|S| = s, |V| = |V'| = m + s$. For any fixed graph $H$ on vertex set $V$, we have 
    \begin{align*}
        &\quad\mathbb{P}\left(G' \text{ is non-even-degenerate}|G = H\right) \\
        &= \mathbb{P}\left(G' \text{ is non-even-degenerate}|G'[S] = H[S]\right)\\
        &\leq c_{m, s}. 
    \end{align*}
    Hence 
    \begin{align*}
        b_{m, s} &= \mathbb{P}\left(G, G' \text{ are both non-even-degenerate}\right)\\
        &\leq c_{m, s}\mathbb{P}\left(G \text{ is non-even-degenerate}\right)\\
        &= c_{m, s}c_{m + s}\\
        &\leq g_{m, s}f_{m + s}\\
        &\leq 16\cdot 2^sf_mf_{m + s}, 
    \end{align*}
    where the last inequality uses Lemma \ref{g** in terms of f*}. Also, Lemma \ref{decreasing f_n} gives $f_m\leq 2^{s/2}f_{m + s}$ from which the desired result follows. 
\end{proof}

\subsubsection{Solving the recursion} \label{recursion solving}

In this subsection we ``solve" the recursion established in Proposition \ref{key}. First, we turn the inequality of Proposition \ref{key} into a bound for $f_{2n}$ in terms of $f_n$.

\begin{lem} \label{key cor}
    For any positive integer $n,a,t$ with $t\leq \frac{2}{3}n$ and $a < t/2$, we have 
    \begin{equation*}
        f_{2n} \leq 600\cdot t^{4a}2^{3a}f_n^2 + 30\cdot 2^{-t/2} + 20\cdot 2^{-a^2}. 
    \end{equation*}
\end{lem}

\begin{proof}
    Let $y = 600\cdot t^{4a}2^{3a}f_n^2 + 30\cdot 2^{-t/2} + 20\cdot 2^{-a^2}$, the right hand side of the desired inequality. 
    
    For any $t\leq m\leq n$, we have $f_m \leq 2^{(n - m)/2}f_n$ by Lemma \ref{decreasing f_n}. Hence applying Proposition~\ref{key} with $m$ in place of $n$, we have 
    \begin{align*}
        c_{2m}&\leq 256\cdot t^{4a}2^{3a}f_m^2 + 14\cdot 2^{-t/2} + 8\cdot 2^{-a^2}\\
        &\leq 2^{n - m}\left(256\cdot t^{4a}2^{3a}f_n^2 + 14\cdot 2^{-t/2} + 8\cdot 2^{-a^2}\right)\\
        &\leq 2^{n - m}\cdot y/2. 
    \end{align*}Hence $c_{2m}2^{(2m - 2n)/2}\leq y/2$. By Proposition \ref{decreasing c_n}, $c_{2m + 1}\leq c_{2m} + 2^{-2m}$, hence $c_{2m + 1}2^{(2m + 1 - 2n)/2}\leq \sqrt{2}\left(y/2 + 2^{- m - n}\right)\leq y$ as $y\geq 30\cdot 2^{-t/2}\geq 30\cdot 2^{-m - n}$. Therefore $c_{n'}2^{(n' - 2n)/2}\leq y$ for all $2t\leq n'\leq 2n$. 
    
    For $n'\leq 2t$, $c_{n'}2^{(n' - 2n)/2}\leq 2^{(n' - 2n)/2}\leq 2^{t - n}\leq 2^{-t/2}\leq y$. Hence $c_{n'}2^{(n' - 2n)/2}\leq y$ for all $n'\leq 2n$. We have also that $2^{-(2n)/6}\leq f_n^2\leq y$, hence $f_{2n} = \max\left(2^{-(2n)/6}, \max_{n'\leq 2n}c_{n'}2^{(n' - 2n)/2}\right)\leq y$ as needed. 
\end{proof}

\begin{df}
    Let $h_n = -\log_2 f_n$.
\end{df}

In terms of $h_n$, our aim to prove that $c_n = e^{-\Omega(n)}$ reduces to showing $h_n = \Omega(n)$. We first bound the growth of $h_n$ using the analogous result (Lemma \ref{decreasing f_n}) for $f_n$. 

\begin{lem} \label{increasing h_n}
    For any $n\geq 1$, we have $h_{n + 1}\leq h_n + 1/2$. Furthermore, for any $m\geq n\geq 2$, we have $h_m\geq h_n - 1$. 
\end{lem}

\begin{proof}
    By Lemma \ref{decreasing f_n}, we have $f_{n+1}\geq 2^{-1/2}f_n$. Taking the logarithm gives the first result. 

    To prove the second part of the assertion, note that Lemma \ref{decreasing f_n} implies $f_m\leq f_n + 2^{-n + 1}\leq 2 f_n$ as $f_n\geq 2^{-n/6}\geq 2^{-n + 1}$. Taking the logarithm gives $h_m\geq h_n - 1$ as needed. 
\end{proof}

\begin{lem} \label{key nice}
    For any $n$ with $h_n\geq 4$, we have $h_{2n}\geq 2h_n - 22\sqrt{h_n}\log_2 h_n - 10$. 
\end{lem}

\begin{proof}
    Let $a\geq 3$ and let $t = a^2$. Lemma \ref{key cor} gives 
    \begin{align*}
        f_{2n}&\leq 600\cdot a^{8a}2^{3a}f_n^2 + 30\cdot 2^{-a^2/2} + 20\cdot 2^{-a^2}\\
        &\leq 600\cdot a^{11a}f_n^2 + 50\cdot 2^{-a^2/2}
    \end{align*} whenever $\frac{3}{2}a^2\leq n$. 

    For any $n$ satisfying $h_n\geq 4$, we may take $a = \lfloor{\sqrt{4h_n}}\rfloor$. Indeed, then we have $a\geq 3$ and $\frac{3}{2}a^2\leq \frac{3}{2}4h_n\leq n$, where the last inequality follows from $f_n\geq 2^{-n/6}$. With this choice of $a$, we have 
    \begin{align*}
        f_{2n}&\leq 600\cdot a^{11a}f_n^2 + 50\cdot 2^{-(a + 1)^2/2}2^{(2a + 1)/2}\\
        &\leq 600\cdot a^{11a}f_n^2 + 50\cdot 2^{-2h_n}2^{(2a + 1)/2}\\
        &= \left(600\cdot a^{11a} + 50\cdot 2^{(2a + 1)/2}\right)f_n^2\\
        &\leq 650\cdot a^{11a}f_n^2. 
    \end{align*}
    Taking the logarithm of both sides, we have $h_{2n}\geq 2h_n - 11\sqrt{4h_n}\log_2\left(\sqrt{4h_n}\right) - 10\geq 2h_n - 22\sqrt{h_n}\log_2 h_n - 10$.
\end{proof}

\begin{lem} \label{stronger lem}
    If $h_{n_0}\geq 2^{50}$ for some $n_0$, then $h_n = \Omega(n)$. 
\end{lem}

\begin{proof}
    We observe that $25\log_2 x\leq x^{1/4}$ for all $x\geq 2^{50}$, and $x - 10\cdot x^{3/4}$ is an increasing function in $x$ for $x\geq 2^{50}$. 
    
    For any $n$ with $h_n\geq 2^{50}$, by Lemma \ref{key nice}, we have 
    \begin{align*}
        h_{2n}&\geq 2h_n - 22\sqrt{h_n}\log_2 h_n - 10\\
        &\geq 2h_n - 25\sqrt{h_n}\log_2 h_n\\
        &\geq 2h_n - h_n^{3/4}. 
    \end{align*}
    Hence for any such $n$,
    \begin{align*}
        h_{2n} - 10\cdot h_{2n}^{3/4}&\geq \left(2h_n - h_n^{3/4}\right) - 10\cdot\left(2h_n - h_n^{3/4}\right)^{3/4}\\
        &\geq 2h_n - h_n^{3/4} - 10\cdot 2^{3/4}h_n^{3/4}\\
        &\geq 2\cdot (h_n - 10\cdot h_n^{3/4}). 
    \end{align*}
    As $h_{2n}\geq h_n\geq 2^{50}$ as well, we may iterate this, giving $h_{2^kn_0} - 10\cdot h_{2^kn_0}^{3/4}\geq 2^k\cdot (h_{n_0} - 10\cdot h_{n_0}^{3/4})$ for any non-negative integer $k$ and $n_0$ with $h_{n_0}\geq 2^{50}$. Hence $h_{2^kn_0}\geq 2^k\cdot (h_{n_0} - 10\cdot h_{n_0}^{3/4})\geq 2^k$ for any non-negative integer $k$. By the second assertion of Lemma \ref{increasing h_n}, this implies that $h_n = \Omega(n)$. 
\end{proof}

\begin{prop}
    We have $c_n = e^{-\Omega(n)}$. 
\end{prop}

\begin{proof}
    By Proposition \ref{o(1) prop}, we have $c_n = o(1)$. Thus $f_n \leq \max(c_n, \dots, c_{\lfloor{n/2}\rfloor}, 2^{-n/6}) = o(1)$ and therefore $h_n \rightarrow \infty$ as $n\rightarrow \infty$. Hence $h_{n_0}\geq 2^{50}$ for some $n_0$. Therefore by Lemma \ref{stronger lem}, we have $h_n = \Omega(n)$. Thus $c_n \leq f_n = e^{-\Omega(n)}$. 
\end{proof}

\section{Concluding remarks}

In this paper we proved (in Theorem \ref{thm:degenerate}) that the probability that $G(n,1/2)$ is not even-degenerate is $e^{-\Theta(n)}$. Our argument relied on the edge probability being exactly $1/2$. It would be interesting to extend this result to a general $G(n,p)$, at least when $p$ is a constant.

\noindent \textbf{Acknowledgements.} We are grateful to Noga Alon and Leo Versteegen for helpful comments and suggestions.

\bibliographystyle{abbrv}
\bibliography{mybib}

\begin{thebibliography}{10}

\bibitem{alon2024graph}
N.~Alon.
\newblock Graph-codes.
\newblock {\em European Journal of Combinatorics}, 116:103880, 2024.

\bibitem{Alonpersonal}
N.~Alon, personal communication, 2024.

\bibitem{alon2023structured}
N.~Alon, A.~Gujgiczer, J.~K{\"o}rner, A.~Milojevi{\'c}, and G.~Simonyi.
\newblock Structured codes of graphs.
\newblock {\em SIAM Journal on Discrete Mathematics}, 37(1):379--403, 2023.

\bibitem{heath2023edge}
P.~Bennett, E.~Heath, and S.~Zerbib.
\newblock Edge-coloring a graph {$ G $} so that every copy of a graph {$ H $} has an odd color class.
\newblock {\em arXiv preprint arXiv:2307.01314}, 2023.

\bibitem{berger2023k4}
A.~Berger and Y.~Zhao.
\newblock {$K_4$}-intersecting families of graphs.
\newblock {\em Journal of Combinatorial Theory, Series B}, 163:112--132, 2023.

\bibitem{cameron2023new}
A.~Cameron and E.~Heath.
\newblock New upper bounds for the {E}rd{\H{o}}s-{G}y{\'a}rf{\'a}s problem on generalized {R}amsey numbers.
\newblock {\em Combinatorics, Probability and Computing}, 32(2):349--362, 2023.

\bibitem{ellis2012triangle}
D.~Ellis, Y.~Filmus, and E.~Friedgut.
\newblock Triangle-intersecting families of graphs.
\newblock {\em Journal of the European Mathematical Society (EMS Publishing)}, 14(3), 2012.

\bibitem{ge2023new}
G.~Ge, Z.~Xu, and Y.~Zhang.
\newblock A new variant of the {E}rd{\H{o}}s--{G}y{\'a}rf{\'a}s problem on {$K_5$}.
\newblock {\em arXiv preprint arXiv:2306.14682}, 2023.

\bibitem{gowers}
W.~T. Gowers.
\newblock The first unknown case of polynomial {DHJ}.
\newblock \url{https://gowers.wordpress.com/2009/11/14/the-first-unknown-case-of-polynomial-dhj/}, blog post, 2009.

\bibitem{versteegen2023upper}
L.~Versteegen.
\newblock Upper bounds for linear graph codes.
\newblock {\em arXiv preprint arXiv:2310.19891}, 2023.

\end{thebibliography}

\end{document}